\documentclass{tac}

\usepackage{graphicx}
\usepackage{amsmath,amssymb,url}
\usepackage{tikz}
\usepackage{tikz-cd}
\usetikzlibrary{arrows,calc,decorations,decorations.markings,fadings,positioning,patterns,shapes
	,decorations.pathmorphing,backgrounds,fit,petri,mindmap,shadows,positioning,matrix}
\usepackage{multirow,tabularx,makecell}
\usepackage [all]{xy}
\usepackage{pigpen}
\newcommand{\xyline}[2][]{\ensuremath{\smash{\xymatrix@1#1{#2}}}}
\makeatother
\newdir{ >}{{}*!/-7.5pt/@{>}}
\makeatletter
\xyoption{2cell}
\UseAllTwocells

\newcommand{\po}{\ar@{}[dr]|{\text{\pigpenfont R}}}
\newcommand{\pb}{\ar@{}[dr]|{\text{\pigpenfont J}}}

\def\Cal#1{{\cal#1}}

\newtheorem{theorem}{Theorem}
\newtheorem{definition}[theorem]{Definition}
\newtheorem{proposition}[theorem]{Proposition}
\newtheorem{lemma}[theorem]{lemma}

\newcommand{\wquot}{/\!\!/}

\newcommand{\ra}{\rightarrow}

\tikzset{->-/.style={decoration={
			markings,
			mark=at position .5 with {\arrow{>}}},postaction={decorate}}}

\address{Department of Electrical Engineering,\\
	Indian Institute of Technology Delhi, \\
	New Delhi-110016, INDIA.\\[5pt]
	Department of Electrical Engineering,\\
	Indian Institute of Technology Delhi, \\
	New Delhi-110016, INDIA.\\
}
\eaddress{salil.samant@ee.iitd.ac.in\CR sdjoshi@ee.iitd.ac.in}

\title{Unified Functorial Signal Representation I: From Grothendieck fibration to Base structured categories.}
\author{Salil Samant and Shiv Dutt Joshi}

\begin{document}

\maketitle
\begin{abstract}
	In this paper we study categories $(F,\mathbf{C},\mathbf{D})$ and $(\mathbb{F},\mathbf{C},\mathbf{Set})$ as graphs of a functor and prove these to be fibred on $\mathbf{C}$. Then we examine Grothendieck construction in the context of an ordinary functor $F: \mathbf{C} \rightarrow \mathbf{D}$ through the concept of trivial categorification, using an appropriate functor $\mathbf{F}: \mathbf{C} \xrightarrow{F} \mathbf{D} \xrightarrow{I} \mathbf{Cat}$  to construct $\int_{\mathbf{C}^{op}} \bar{\mathbf{F}}$. This category characterizes a functor as an abstract right category action while its dual $\mathcal{X} \rtimes_{\mathbf{F}} \mathbf{C}$ or $(\int_{\mathbf{C}^{op}} \bar{\mathbf{F}})^{op}$ characterizes a functor as an abstract left category action. Similarly using $\mathbb{F}: \mathbf{C} \xrightarrow{F} \mathbf{D} \xrightarrow{U} \mathbf{Set}$ we define $\mathcal{X} \rtimes_{\mathbb{F}} \mathbf{C}$ and ${\int_{\mathbf{C}^{op}} \bar{\mathbb{F}}}$ as categories denoting concrete left and right actions of $\mathbf{C}$ respectively. Collectively referred to as `base structured categories', these are proven abstractly isomorphic to the base category $\mathbf{C}$ but concretely isomorphic to each other or $(\mathbb{F},\mathbf{C},\mathbf{Set}) \cong  \int_{\mathbf{C}} \bar{\mathbb{F}} \cong \mathcal{X} \rtimes_{\mathbb{F}} \mathbf{C}$. These are special instances of fibred categories where the base category is $\mathbf{C}$ and fibres are $\mathbf{D}$ objects being treated as trivial categories. The perspective of making only base structure explicit through category theory concealing the vertical structure using identity morphisms enables one to combine intuitions of Grothendieck's relative and Leyton's generative theory. As explored further, it greatly facilitates the application of functors in certain fundamental applications which hitherto have been treating objects of category $\mathbf{D}$ purely in a set theoretic way.
	
\end{abstract}

\section{Introduction}
\label{intro}

Here we introduce and study six categories forming a kind of family for which we have coined the term base structured categories. These precise mathematical expressions characterizing a functor in certain distinct ways enable us to exploit both the category and (structured) set theoretic perspectives simultaneously utilizing advantages of both these theories immensely from an application viewpoint. Although from a pure category theory perspective it can be argued that these don't yield new data apart from that already contained in the functor except for the new perspective of treating objects as trivial categories. However we demonstrate that such a concept of looking at ordinary objects as trivial categories in-fact strengthens the possibility of applying category theory (at least partly) where traditionally objects are being treated in set-theoretic manner without realizing category theory lurking beneath the surface in applications. Moreover using these trivial categories we can form precise mathematical expressions such as $\mathcal{X} \rtimes_{\mathbf{F}} \mathbf{C}$ where $\mathcal{X} = \mathbf{F}X \amalg \mathbf{F}Y \amalg ... $ is the coproduct of objects considered as trivial categories which always exists in $\mathbf{Cat}$ even when the coproduct need not exist in the $\mathbf{D}$ category truly generating the possibility of characterizing the functor as multi-object action of a general category explicitly as semidirect product of purely categories. The deep significance of this distinct perspective is motivated in Section \ref{motiv}. For instance being obviously isomorphic (abstractly) to $\mathbf{C}$ it seems that $(F,\mathbf{C},\mathbf{D})$ is no more interesting than $\mathbf{C}$; however as Proposition~\ref{prop} proves that it is concretely different from $\mathbf{C}$ and therefore possibly carries additional set-theoretic structure in addition to being a category. This creates possibility of acknowledging the fact that applications where mathematical constructions have functorial properties can benefit from using both category theory (with well-known strength of relative perspective or analogy and comparison) and set theory (with the well-known strength of computation on elements within the objects) simultaneously through the distinct perspective of base structured categories as studied in this paper. The perspective is roughly summarized in Table~\ref{perspective}.       

\begin{center}
	\begin{table}[ht]
		\label{perspective}
		\centering
		\resizebox{\linewidth}{!}{
			\begin{tabular}{|c|c|c|}\hline
				\textbf{Fibred Category} & \textbf{Base Structured Category}  \\ \hline
				Pure category theory & Partly category theory and $\mathbf{D}$-theory (structured sets) \\ \hline
				Arrows decomposed into vertical and Cartesian & All arrows Cartesian (trivial vertical arrows)  \\ \hline
				Arrows treated fundamental & Objects treated in set theory while arrows depict relativity  \\ \hline
				Complete categorification ($\mathbf{D}$ treated as 2-category) & Partial categorification ($\mathbf{D}$ treated as trivial 1-category)  \\ \hline
				Fibres are pure categories & Fibres are structured sets viewed as trivial categories \\ \hline
			\end{tabular}
		}
		\caption{Heuristic of Base structured category combining category theory with set theory}
	\end{table}
\end{center}
In applications such as signal representation where classical set-theoretic Hilbert and Banach theory is being used; this perspective (the graph of a functor) made us realize that categories are implicitly involved and the fundamental concept of redundancy directly affecting true information within a signal simply follows from the relative point of view as studied comprehensively in~\cite{salilp3}. Some other applications involving set-theoretic actions and symmetry also benefit from the perspective (of transformation categories) and are explored in~\cite{salilp2}. The abstract pertaining to signal representation was presented earlier in CT 2016 \cite{salilct2016} and expanded in~\cite{salilp3}.

\subsection{Structure and Contribution}

Two natural motivations which led to mathematical formulation of these categories  are discussed at length in Section~\ref{motiv}. It serves as an overview and justifies their independent study. The category characterizing a functor as a structure preserving morphism is defined in Section~\ref{fcd} which we have termed abstract graph of a functor. It is shown as both a fibration and an opfibration in Proposition~\ref{prop2}. A related concrete version is formulated using Definition~\ref{defn2} and proved as opfibration in Proposition~\ref{prop3}. As far as we understand at this point, the terminology might conflict with existing notion of graph of a functor for a reason discussed in Section~\ref{sec:graph}.

In the next Section~\ref{act}, we make it precise how base structured categories generalize the familiar transformation groupoid using Proposition~\ref{prop4} justifying terminology of transformation categories. First we examine the classic Grothendieck construction in the context of an ordinary (1-)functor $F: \mathbf{C} \rightarrow \mathbf{D}$ through a suitable composition as $\mathbf{F}: \mathbf{C} \xrightarrow{F} \mathbf{D} \xrightarrow{I} \mathbf{Cat}$ using the precise concept of trivial categorification discussed in Section~\ref{trivcat}. Using this we construct a category $\int_{\mathbf{C}^{op}} \bar{\mathbf{F}}$ as defined in~\ref{defn3} denoting right abstract action which could be thought of as the left action of $\mathbf{C}^{op}$ . Definition~\ref{defn4} correspondingly defines a category denoting concrete right action of category on the underlying sets of objects. The left action counterparts are formulated as Definition~\ref{defn5} and Definition~\ref{defn6} respectively. The duality between categories defined as right actions and left actions is discussed in Section~\ref{dual}. Given an ordinary functor it is possible to construct a category $\int_{\mathbf{C}^{op}} \bar{\mathbf{F}}$. Its opposite is a category $(\int_{\mathbf{C}^{op}} \bar{\mathbf{F}})^{op}$ seen as identical to $\mathcal{X} \rtimes_{\mathbf{F}} \mathbf{C}$. Under the condition $\mathbf{C} \cong \mathbf{C}^{op}$, the category ${\int_{\mathbf{C}} \bar{\mathbf{F}}}$ which denotes functor as a abstract multi-object right category action is also isomorphic to abstract left action. 

Finally we make use of notions defined earlier to illustrate that base structured categories being categorical fibrations are abstractly isomorphic to the base but concretely to each other culminating in Proposition~\ref{prop}.

All basic notions of category theory utilized in the main body of paper are recalled at the end in Appendix~\ref{subsec:cf},~\ref{sec:pull},~\ref{fib1}. 

\subsection{Motivation}
\label{motiv}

There are two natural yet significant motivations to introduce and study six categories viz. $(F,\mathbf{C},\mathbf{D})$, $(\mathbb{F},\mathbf{C},\mathbf{Set})$, $\mathcal{X} \rtimes_{\mathbf{F}} \mathbf{C}$, $\int_{\mathbf{C}^{op}} \bar{\mathbf{F}}$, $\mathcal{X} \rtimes_{\mathbb{F}} \mathbf{C}$ and $\int_{\mathbf{C}^{op}} \bar{\mathbb{F}}$. The categories characterizing a functor as a structure preserving morphism cater to the purpose of signal representation  while those characterizing a functor as an action of category on objects serve to express actions of various mathematical structures (at least partly through trivial categorification) as semidirect product of categories. These semidirect (or Grothendieck completion) expressions meaningfully hold even in cases where coproduct of a family of objects on which the action takes place need not exist in codomain category. These motivations are as follows:   

\begin{enumerate}

\item \textbf{A unified category theoretic perspective of all signal spaces and precise arrow-theoretic mathematical formulation of redundancy within signals} 

A surprisingly trivial construction $(F,\mathbf{C},\mathbf{D})$ (being very obviously isomorphic to category $\mathbf{C}$) for the purposes of this sequel offers a distinctly refreshingly new perspective on any ordinary functor ${F}: \mathbf{C} \rightarrow \mathbf{D}$. It is the mathematical expression which combines the intuitive generative perspective of Leyton's Theory in psychology~\cite{Leyton86a},~\cite{Leyton01} along with well-known Grothendieck's relative point of view~\cite{grorel} by treating it as a fibration. It leads to formulation of category $(\mathbb{F},\mathbf{C},\mathbf{Set})$ which can be viewed as opfibration. The distinct perspective in a nutshell is that while general fibred categories carry both structures i.e from base category as well as fibre categories through Cartesian and vertical arrows, the base structured categories carry only the base structure concealing the vertical structure within the objects since vertical arrows are just identities. 

This perspective can be utilized in the context of signal representation where we can view a signal space as being fibred on some category $\mathbf{B}$ that captures natural generative mechanism of a given signal to be represented (the objects are treated as generators while arrows capture the transfer of Leyton's theory). Then intra-signal redundancy gets naturally modeled as Cartesian lift of the base arrows. Going by the philosophy of category theory which treats arrows as more fundamental to objects these categories exploit the generative base $\mathbf{B}$ structure of a signal explicitly in category theory while the vertical $\mathbf{D}$ structure (often Banach or Hilbert or Riesz space) hidden from categorical formulation can be treated in set-theory as usual at the local level of objects utilizing benefits of both at application level. In essence the signal space can be treated as a category $(F,\mathbf{C},\mathbf{D})$ emphasizing the generative structure of a signal.

Although this is studied in detail~\cite{salilp3} here we briefly discuss a simple prototypical example of a music signal as shown in Figure~\ref{fig:music_signal} to show how this is put to practical application in a signal carrying simple translational redundancies.

\begin{figure}
	\resizebox*{\textwidth}{!}{
\begin{tikzpicture}
\draw[domain=-1.57:1.57,samples=100] plot(\x,{sin(2*\x r)});
\draw[domain=7.85:11,samples=100] plot(\x,{sin(2*\x r)});
\draw plot [smooth] coordinates {(1.57,0) (3,-0.5) (4,0.75) (5,1) (6,-0.4) (7,1) (7.85,0)};
\draw[->] (-4,0) -- (13,0) node [pos=1,below] {$t$};
\draw plot [smooth] coordinates {(-1.57,0) (-2.5,1) (-2.8,0) (-3.5,-0.5) (-3.7,-1) (-4,0)};
\draw plot [smooth] coordinates {(11,0) (11.5,-0.5) (12,1) (12.5,0.5)};
\draw[->] (-3.5,-1) --(-3.5,1.5) node [pos=1,left] {$\mathbb{R}$};
\draw[-,dashed] (-1.57,1.5) -- (-1.57,-1.5);
\draw[-,dashed] (1.57,1.5) -- (1.57,-1.5);

\draw[-,dashed] (7.85,1.5) -- (7.85,-1.5);
\draw[-,dashed] (11,1.5) -- (11,-1.5);
\draw[<->] (-1.57,-1.3) --(1.57,-1.3) node[midway,fill=white] {$I$};
\draw[<->] (7.85,-1.3) --(11,-1.3) node[midway,fill=white] {$J$};
\node at (0,1.3) [draw=none] {$f$};
\node at (9.4,1.3) [draw=none] {$f'$};
\node at (-3.5,0) [draw=none,above left] {$0$};
\end{tikzpicture}
}
\caption{A global music signal where the local signal $f=F(M)$ in interval $I$ is transfered(translated) to $f'=F(M')$ in interval $J$}
\label{fig:music_signal}
\end{figure}
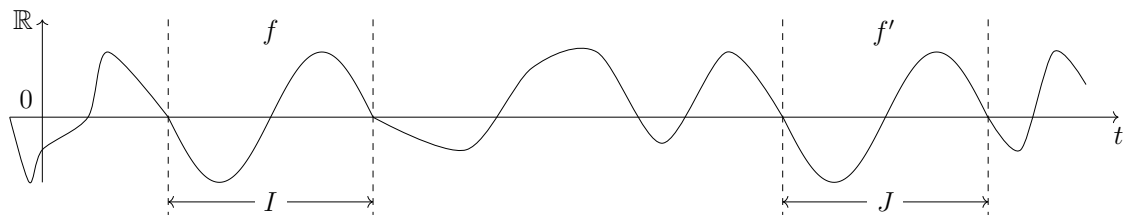

A music signal has a sheet-music generative structure exhibiting several local melody compositions and their transfers introducing translational redundancies within the signal. One could consider a simple category structure on sheet music where symbols (or their combination) are objects and there are arrows between identical symbols denoting translation. Now consider specific instance of the melody $M$ being transfered to $M'$ within entire composition then corresponding signals $f$ and $f'$ generated also preserve this structure as captured by a functor $F$. 

\begin{equation}
\xymatrix{
	M \ar[r]^{T} & M'
} 
\quad \\
\xrightarrow{F}
\quad
\xymatrix{
	(\Bbb R,\Sigma_{\Cal B}) \ar[r]_{id} & (\Bbb R,\Sigma_{\Cal B})  \\
	(I,\Sigma_{I}) \ar[r]_{g} \ar@{->}[u]^{f = F(M)} & (J,\Sigma_{J}) \ar@{->}[u]_{f' = F(M')}
}
\end{equation}

Now $(I,\Sigma_I,\mu)$  and $(J,\Sigma_J,\nu)$ are 
measure spaces and $g :I \to J$ is an inverse-measure-preserving function.  
Also $f'$ is a $\nu$-virtually measurable $[-\infty,\infty]$-valued function 
defined on $J$, then $f'g = f$ is $\mu$-virtually measurable. It can be shown that 
$f'\mapsto f'g: L^2(\nu)\to L^2(\mu)$ induces a linear operator
$T:L^2(\nu)\to L^2(\mu)$ such that $\|Tv\|_2=\|v\|_2$ for every
$v\in L^2(\nu)$. This naturally suggests a signal space which respects or is
compatible with this structure. Such a signal space which is matched to natural generative structure and inter-relationships
between these generators takes the form of base structured category $(L^2|_{\mathbf{B}}, \mathbf{B}, \mathbf{Hilb}) \rightarrow \mathbf{B}$ as compared to classic generically fixed space $L^2(R)$ where $\mathbf{B}$ captures natural generative structure of signal in sense of Leyton's theory~\cite{Leyton01}. For a faithful functor this matched space is isomorphic to image subcategory $L^2({\mathbf{B}})$ and matched representation in this category as signal space is given by Equation~\ref{rep} (using contravariant form of $L^2$ functor on measure spaces which is same as covariant form on measure algebras; see~\cite{salilp3}) :

\begin{equation}
\label{rep}
Signal=
\underbrace{f                       \rule[-12pt]{0pt}{5pt}}_{\mbox{classic}}
+\underbrace{f' - L^2(g^{-1})f                \rule[-12pt]{0pt}{5pt}}_{\mbox{relative error}}
+\underbrace{L^2(g^{-1})f \rule[-12pt]{0pt}{5pt}}_{\mbox{generative relative term}}
+  ...,
\end{equation}

\begin{itemize}
	\item $f$ : Local signal Representation using standard basis in $L^2(I)$
	
	\item $f'$ : Local signal Representation using standard basis in $L^2(J)$
	
	\item $L^2(g^{-1})f$ : Transformed/Transfered local signal from $L^2(I)$ to $L^2(J)$
	
	\item $f' - L^2(g^{-1})f$ : True innovation or Difference between transformed and measured local signal in $L^2(J)$.

\end{itemize}

We discuss in~\cite{salilp3} that this could be broadly generalized to prove that such a category theoretic formulation mathematically models various existing differential encoding standards of signal and image representation, that are presently not understood in category-theoretic way.

\item \textbf{Transformation Categories generalizing transformation groupoid $X \wquot G$}
In the spirit of a group action on a set giving rise to a transformation groupoid; the base category action on the underlying sets of D-objects (in case of $\mathbf{D}$ being construct) gives rise to categories $\mathcal{X} \rtimes_{\mathbb{F}} \mathbf{C}$ and $\int_{\mathbf{C}^{op}} \bar{\mathbb{F}}$ which capture the concrete action of base arrows on the underlying elements of the D-objects. These base structured categories are then motivated as the generalization of classic transformation groupoids suggesting the terminology {\bf Transformation Categories}.

Now the semidirect product of groups (which is a group), semidirect product of monoids (which is again a monoid), semidirect product of categories (which is again a category) and of a set (treated as a discrete category) and a category (which obviously results into a category often termed as transformation groupoid). All of these are just semidirect products of purely categories since the structures of a  group, monoid, set are all captured using a category. But what about the case of objects of an arbitrary category $\mathbf{D}$ which cannot be easily treated as categories? It is visually seen using commutative diagrams that even in an arbitrary case of a functor ${F}: \mathbf{C} \rightarrow \mathbf{D}$ clearly there is an induced action of arrows. This is the motivation for introducing new meaningful mathematical expressions such as $\mathcal{X} \rtimes_{\mathbf{F}} \mathbf{C}$ and $\int_{\mathbf{C}^{op}} \bar{\mathbf{F}}$ which make implicit action in a functor precisely explicit at least partly by treating objects as trivial categories. The action at level of object is only partially captured by category theory in the sense that object after being transformed by arrows is still viewed as same object using trivial category. However note that when the object undergoes an internal state transformation this change can be captured at level of elements within the object considered as structured set. Thus realizing that actions at global (or outer level) are being captured category-theoretically while actions at local (or inner) object level implicitly remain at set-theoretic level one actually makes use both these theories simultaneously for application purposes especially in symmetries and geometries. We motivate this point in detail for more clarity.

A semidirect product of two groups $G$ and $H$ is the same as a group homomorphism $\phi: G \rightarrow Aut(H)$ i.e. an action of $G$ on $H$ and denoted as $H \rtimes_{\phi} G$. In category theory this is same as Grothendieck construction of $F: \mathbf{G} \rightarrow \mathbf{Grp}$ where $\mathbf{G}$ is a category with a single object $\ast$ and morphisms $Hom(\ast,\ast) = G$. The category $\mathbf{Grp}$ is a category of all groups and group homomorphisms while $F(\ast) = H$. Thus this action is nothing but the transformation of the object $H$ by arrows $G$. Categorically $\mathbf{H} \rtimes_{\mathbf{F}} \mathbf{G}$ and visually,
\[
\xymatrix{
	\textbf{G} \ar^{F} [r] &\textbf{Grp} \ar^{I} [r] &\textbf{Cat}
}
\]
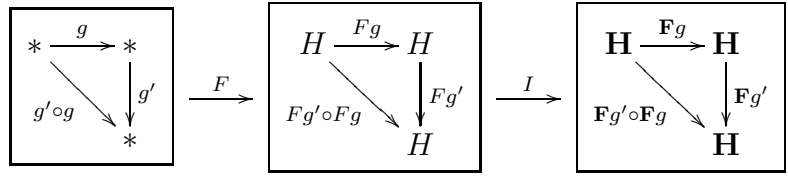
\[
\boxed{\xymatrix{
		\ast \ar_{g' \circ g}[dr] \ar^{g}[r] & \ast \ar^{g'}[d] \\
		& \ast 
}}
\xymatrix{\ar@<-20pt>^{F}[r] &}
\boxed{\xymatrix{
		H \ar_{Fg' \circ Fg}[dr] \ar^{Fg}[r] & H \ar^{Fg'}[d] \\
		& H 
}}
\xymatrix{\ar@<-20pt>^{I}[r] &}
\boxed{\xymatrix{
		\mathbf{H} \ar_{\mathbf{F}g' \circ \mathbf{F}g}[dr] \ar^{\mathbf{F}g}[r] & \mathbf{H} \ar^{\mathbf{F}g'}[d] \\
		& \mathbf{H} 
}}
\]
\begin{center}
	\begin{table}[ht]
		\centering
		\resizebox{\linewidth}{!}{
			\begin{tabular}{|c|c|c|c|}\hline
				Group theoretic action & Classic Semidirect Product & Functor as action & Category-theoretic Semidirect Product \\ \hline
				$\phi : G \rightarrow Aut(H)$ & $H \rtimes_{\phi} G$ & $\mathbf{F}: \mathbf{G} \rightarrow \mathbf{Grp} \rightarrow \mathbf{Cat}$ & $\mathbf{H} \rtimes_{\mathbf{F}} \mathbf{G}$ \\ \hline
			\end{tabular}
		}
		\caption{Group action on a Group - Set theoretic versus Category theoretic}
		\label{table}
	\end{table}
\end{center}

Likewise a semidirect product of two monoids $M$ and $N$ is the same as a monoid homomorphism $\phi: M \rightarrow End(N)$ i.e. an action of $M$ on $N$ and denoted as $N \rtimes_{\phi} M$. In category theory this is same as Grothendieck construction of $F: \mathbf{M} \rightarrow \mathbf{Mon}$ where $\mathbf{M}$ is a category with a single object $\ast$ and morphisms $Hom(\ast,\ast) = M$. The category $\mathbf{Mon}$ is a category of all monoids and monoid homomorphisms while $F(\ast) = N$. Thus this is nothing but the action on the object $N$ by arrows $M$. Categorically $\mathbf{N} \rtimes_{\mathbf{F}} \mathbf{M}$ and Visually,
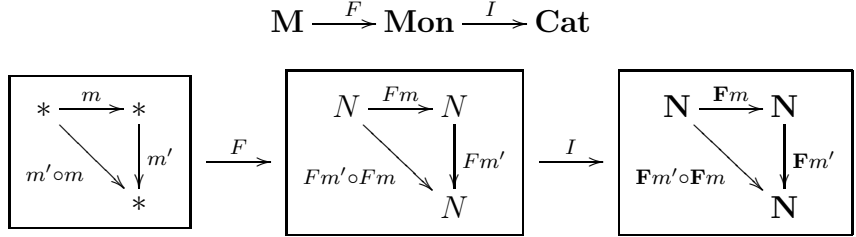
\[
\xymatrix{
	\textbf{M} \ar^{F} [r] &\textbf{Mon} \ar^{I} [r] &\textbf{Cat}
}
\]
\[
\boxed{\xymatrix{
		\ast \ar_{m' \circ m}[dr] \ar^{m}[r] & \ast \ar^{m'}[d] \\
		& \ast 
}}
\xymatrix{\ar@<-20pt>^{F}[r] &}
\boxed{\xymatrix{
		N \ar_{Fm' \circ Fm}[dr] \ar^{Fm}[r] & N \ar^{Fm'}[d] \\
		& N 
}}
\xymatrix{\ar@<-20pt>^{I}[r] &}
\boxed{\xymatrix{
		\mathbf{N} \ar_{\mathbf{F}m' \circ \mathbf{F}m}[dr] \ar^{\mathbf{F}m}[r] & \mathbf{N} \ar^{\mathbf{F}m'}[d] \\
		& \mathbf{N} 
}}
\] 

\begin{center}
	\begin{table}[ht]
		\centering
		\resizebox{\linewidth}{!}{
			\begin{tabular}{|c|c|c|c|}\hline
				Monoid theoretic action & Classic Semidirect Product & Functor as action & Category-theoretic Semidirect Product \\ \hline
				$\phi : M \rightarrow End(N)$ & $N \rtimes_{\phi} M$ & $\mathbf{F}: \mathbf{M} \rightarrow \mathbf{Mon} \rightarrow \mathbf{Cat}$ & $\mathbf{N} \rtimes_{\mathbf{F}} \mathbf{M}$ \\ \hline
			\end{tabular}
		}
		\caption{Monoid action on a Monoid - Set theoretic versus Category theoretic}
		\label{table}
	\end{table}
\end{center}  

The generalization of this for categories is well known as the semidirect product of categories. This was formulated by Grothendieck in~\cite{SGA1} and widely known as Grothendieck construction of $F: \mathbf{C} \rightarrow \mathbf{Cat}$ where $\mathbf{C}$ is a general category. The category $\mathbf{Cat}$ is a category of all categories and functors. Thus this is nothing but the action of arrows of category $\mathbf{C}$ on the objects (which are categories) within the image sub-category $F(\mathbf{C})$. Categorically $(FX \amalg FY \amalg ...) \rtimes_{{F}} \mathbf{C}$ and Visually,
\[
\xymatrix{
	\textbf{C} \ar^{F} [r] &\textbf{Cat}
}
\]
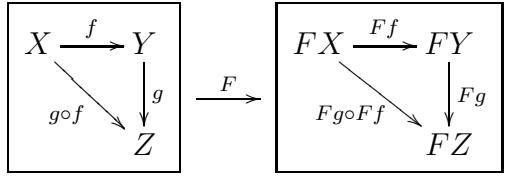
\[
\boxed{\xymatrix{
		X \ar_{g \circ f}[dr] \ar^{f}[r] & Y \ar^{g}[d] \\
		& Z 
}}
\xymatrix{\ar@<-20pt>^{F}[r] &}
\boxed{\xymatrix{
		FX \ar_{Fg \circ Ff}[dr] \ar^{Ff}[r] & FY \ar^{Fg}[d] \\
		& FZ 
}}
\] 
\begin{center}
	\begin{table}[ht]
		\centering
		\resizebox{\linewidth}{!}{
			\begin{tabular}{|c|c|}\hline
				Functor into $\mathbf{Cat}$ as action of category on category & Category-theoretic Semidirect Product \\ \hline
				${F}: \mathbf{C} \rightarrow \mathbf{Cat}$ & $(FX \amalg FY \amalg ...) \rtimes_{{F}} \mathbf{C}$ \\ \hline
			\end{tabular}
		}
		\caption{Category action on a Category}
		\label{table}
	\end{table}
\end{center}

Now because a general functor can be defined across categories with different structures; variants of the action notion could be captured by various functors. As an easy example, a common group action on a set is categorically $\mathbf{X} \rtimes_{\mathbf{F}} \mathbf{G}$ and visually expressed as 
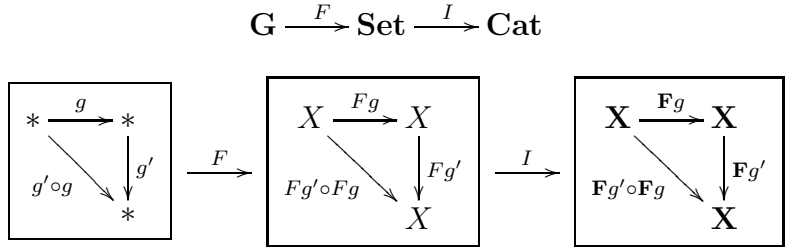
\[
\xymatrix{
	\textbf{G} \ar^{F} [r] &\textbf{Set} \ar^{I} [r] &\textbf{Cat}
}
\]
\[
\boxed{\xymatrix{
		\ast \ar_{g' \circ g}[dr] \ar^{g}[r] & \ast \ar^{g'}[d] \\
		& \ast 
}}
\xymatrix{\ar@<-20pt>^{F}[r] &}
\boxed{\xymatrix{
		X \ar_{Fg' \circ Fg}[dr] \ar^{Fg}[r] & X \ar^{Fg'}[d] \\
		& X 
}}
\xymatrix{\ar@<-20pt>^{I}[r] &}
\boxed{\xymatrix{
		\mathbf{X} \ar_{\mathbf{F}g' \circ \mathbf{F}g}[dr] \ar^{\mathbf{F}g}[r] & \mathbf{X} \ar^{\mathbf{F}g'}[d] \\
		& \mathbf{X} 
}}
\]
\begin{center}
	\begin{table}[ht]
		\centering
		\resizebox{\linewidth}{!}{
			\begin{tabular}{|c|c|c|c|}\hline
				Group theoretic action & Classic Semidirect Product & Functor as action & Category-theoretic Semidirect Product \\ \hline
				$\phi : G \rightarrow Aut(X)$ & -- & $\mathbf{F}: \mathbf{G} \rightarrow \mathbf{Set} \rightarrow \mathbf{Cat}$ & $\mathbf{X} \rtimes_{\mathbf{F}} \mathbf{G}$ or $X \wquot G$ \\ \hline
			\end{tabular}
		}
		\caption{Group action on a Set - Set theoretic versus Category theoretic}
		\label{table}
	\end{table}
\end{center}

This concept of functor as an action is generalized using the example of category action on a set $\mathcal{X}$~\cite{action} which is nothing but a functor ${F}: \mathbf{C} \rightarrow \mathbf{Set}$  where it is shown that action is on the coproduct set $\mathcal{X} =  \amalg_{X \in Ob(\mathbf{C})} F(X)$. This motivates one to consider the action of any ordinary functor not necessarily into $\mathbf{Set}$ as a potential semidirect product of categories. But two main hurdles for an ordinary category $\mathbf{D}$ unlike $\mathbf{Set}$ are categorification of $\mathbf{D}$-objects or in other words finding an inclusion functor into $\mathbf{Cat}$ and existence of coproducts in $\mathbf{D}$. However in the schematic of a functor ${F}: \mathbf{C} \rightarrow \mathbf{D}$ it is clear visually that arrows $f,g..$ induce actions on $FX,FY,..$ through arrows $Ff,Fg,..$, since every arrow in any category could be viewed as the transformation of domain object into a codomain object. This action can be mathematically precise provided one takes the perspective of fibration on $(F,\mathbf{C},\mathbf{D})$ and in the reverse way tries to find a contravariant functor whose Grothendieck construction yields a fibred category at least equivalent to $(F,\mathbf{C},\mathbf{D})$ which is discussed in Section~\ref{act}. Thus in the case of any category $\mathbf{D}$ by treating the objects as trivial categories it becomes possible to express this action as semidirect product of categories precisely as $(\mathbf{F}X \amalg \mathbf{F}Y \amalg ...) \rtimes_{{F}} \mathbf{C}$ since the coproduct is $(\mathbf{F}X \amalg \mathbf{F}Y \amalg ...)$ which always exists being a coproduct of objects in $\mathbf{Cat}$. Now since the objects are treated as trivial categories the action on the individual elements of the underlying sets of objects remains hidden within the objects. However treating $\mathbf{D}$-objects as sets with added axioms one can make this action set-theoretic at level of objects. Alternately in such cases where $\mathbf{D}$ is a concrete category over $\mathbf{Set}$ the action can be made explicit using just pure category theory by utilizing the underlying concrete functor ${U}: \mathbf{D} \rightarrow \mathbf{Set}$ to express the action again as semi-direct product of categories viz $\mathcal{X} \rtimes_{\mathbb{F}} \mathbf{C}$. Using duality or opposite categories we get the dual notion of right actions of categories. 

This motivates us to interpret a general functor $F:\mathbf{C}\rightarrow \mathbf{D}$ as an action; Visually we have 
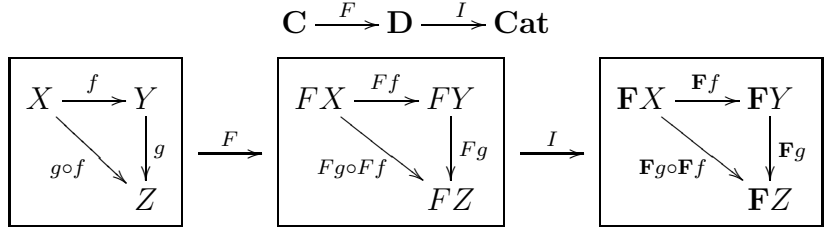
\[
\xymatrix{
	\textbf{C} \ar^{F} [r] &\textbf{D} \ar^{I} [r] &\textbf{Cat}
}
\]
\[
\boxed{\xymatrix{
		X \ar_{g \circ f}[dr] \ar^{f}[r] & Y \ar^{g}[d] \\
		& Z 
}}
\xymatrix{\ar@<-20pt>^{F}[r] &}
\boxed{\xymatrix{
		FX \ar_{Fg \circ Ff}[dr] \ar^{Ff}[r] & FY \ar^{Fg}[d] \\
		& FZ 
}}
\xymatrix{\ar@<-20pt>^{I}[r] &}
\boxed{\xymatrix{
		\mathbf{F}X \ar_{\mathbf{F}g \circ \mathbf{F}f}[dr] \ar^{\mathbf{F}f}[r] & \mathbf{F}Y \ar^{\mathbf{F}g}[d] \\
		& \mathbf{F}Z 
}}
\]   

\begin{center}
	\begin{table}[ht]
		\centering
		\resizebox{\linewidth}{!}{
			\begin{tabular}{|c|c|}\hline
				Ordinary functor as action of category on objects & Category-theoretic Semidirect Product \\ \hline
				${F}: \mathbf{C} \rightarrow \mathbf{D}$ as $\mathbf{F}: \mathbf{C} \rightarrow \mathbf{D} \rightarrow \mathbf{Cat}$ & $(\mathbf{F}X \amalg \mathbf{F}Y \amalg ...) \rtimes_{\mathbf{F}} \mathbf{C}$ \\ \hline
			\end{tabular}
		}
		\caption{Category $\mathbf{C}$ action on Objects $FX,FY,...$ within Image subcategory F$\mathbf{C}$}
		\label{table}
	\end{table}
\end{center}

While the action on underlying sets of objects within image subcategory F$\mathbf{C}$ is motivated by functor $\mathbb{F}: \mathbf{C} \rightarrow \mathbf{D} \rightarrow \mathbf{Set}$ as an action; Visually we have  

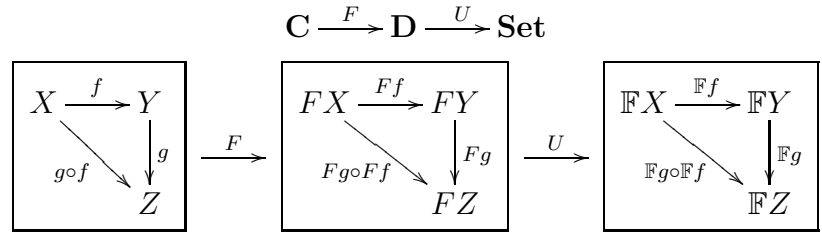
\[
\xymatrix{
	\textbf{C} \ar^{F} [r] &\textbf{D} \ar^{U} [r] &\textbf{Set}
}
\]
\[
\boxed{\xymatrix{
		X \ar_{g \circ f}[dr] \ar^{f}[r] & Y \ar^{g}[d] \\
		& Z 
}}
\xymatrix{\ar@<-20pt>^{F}[r] &}
\boxed{\xymatrix{
		FX \ar_{Fg \circ Ff}[dr] \ar^{Ff}[r] & FY \ar^{Fg}[d] \\
		& FZ 
}}
\xymatrix{\ar@<-20pt>^{U}[r] &}
\boxed{\xymatrix{
		\mathbb{F}X \ar_{\mathbb{F}g \circ \mathbb{F}f}[dr] \ar^{\mathbb{F}f}[r] & \mathbb{F}Y \ar^{\mathbb{F}g}[d] \\
		& \mathbb{F}Z 
}}
\]   

\begin{center}
	\begin{table}[ht]
		\centering
		\resizebox{\linewidth}{!}{
			\begin{tabular}{|c|c|}\hline
				Functor as action of category on underlying sets of objects & Category-theoretic Semidirect Product \\ \hline
				${F}: \mathbf{C} \rightarrow \mathbf{D}$ as $\mathbb{F}: \mathbf{C} \rightarrow \mathbf{D} \rightarrow \mathbf{Set}$ & $(\mathbb{F}X \amalg \mathbb{F}Y \amalg ...) \rtimes_{\mathbb{F}} \mathbf{C}$ \\ \hline
			\end{tabular}
		}
		\caption{Category $\mathbf{C}$ action on Objects $FX,FY,...$ within Image subcategory F$\mathbf{C}$}
		\label{table}
	\end{table}
\end{center}

In essence this perspective motivates us to model a Klein geometry partly category theoretically as $\mathcal{X} \rtimes_{\mathbb{F}} \mathbf{H}$ where manifolds are trivially categorified bringing in a functor underlying geometries. Such a mathematical formulation which works for any arbitrary category $\mathbf{D}$ (in this context category of manifolds) could lead to possibly new formulation of {\bf groupoid geometries} combining category theory and set-theoretic manifold theory  as some sort of multi-object generalization of Klein group geometries since the major hurdle of existence of coproduct objects in $\mathbf{D}$ is eliminated by trivially categorizing objects. This is presently being studied in~\cite{salilp2} and is currently in its draft version.

In summary, the transformation categories denote the action of $\mathbf{C}$ category arrows on arbitrary $\mathbf{D}$-category objects lying within image subcategory $F(\mathbf{C})$; thus every action in mathematics becomes mathematically expressible in form of categorical semidirect product. 

\end{enumerate}

\section{The category $(F,\mathbf{C},\mathbf{D})$ characterizing a functor as a structure preserving morphism}
\label{fcd}
In material set theory, a function between sets is defined as an appropriate subset of the Cartesian product set. This motivates the question if in a similar fashion could a functor be viewed as an appropriate subcategory of a product category ? Given any abstract functor $F: \mathbf{C} \rightarrow \mathbf{D}$ using the commutative diagrams of category $\mathbf{C}$ and image subcategory $F\mathbf{C}$, one could pair and fuse them to construct a new commutative diagram in which the usual axioms of a category are satisfied when everything is done component-wise. This is a subcategory of a product category $\mathbf{C} \times \mathbf{D} $ so it satisfies the axioms of category thus giving us a category characterizing the classic functor.

\begin{definition} [Abstract graph of a functor]
	\label{defn1}
	Consider a (covariant) functor $F:\mathbf{C}\rightarrow \mathbf{D}$. Then $(F,\mathbf{C},\mathbf{D})$ or $\mathbf{G}_F$ is a category consisting of:
	\begin{itemize}
		\item \textbf{objects}: a collection $(X,FX),(Y,FY), ...$ denoted by $\mbox{Ob}(\mathbf{G}_F)$
		
		\item \textbf{morphisms}: a collection $\mathbf{G}_F((X,FX),(Y,FY))=\{(f,Ff):(X,FX)\rightarrow (Y,FY)\}$
		
		\item \textbf{identity}: for each $(X,FX)$, the morphism $\mathbf{1}_{(X,FX)} = (\mathbf{1}_X,\mathbf{1}_{FX})$
		
		\item \textbf{composition}: if $(g,f)\mapsto g\circ f$ in $\mathbf{C}$ then $((g,Fg),(f,Ff))\mapsto (g,Fg)\circ (f,Ff)=(g\circ f,Fg\cdot Ff)$
		
		\item \textbf{unit laws}: for $(f,Ff)$ we have $\mathbf{1}_{(Y,FY)}\circ (f,Ff)=(f,Ff)=(f,Ff)\circ \mathbf{1}_{(X,FX)}$
		
		\item \textbf{associativity}:  $(h,Fh)\circ ((g,Fg)\circ (f,Ff))=((h,Fh)\circ (g,Fg))\circ (f,Ff)$
	\end{itemize}
\end{definition}

The schematic representation of a functor $F: \mathbf{C} \rightarrow \mathbf{D}$ can be denoted using commutative diagrams~\cite{CWM} as  
\begin{equation*}
\xymatrix @R=0.4in @C=0.8in{
	X \ar@(dl,ul)[]|{id_X} \ar[r]^{f} \ar[d]_{h \circ g \circ f} \ar[dr]_{g \circ f} & Y \ar@(dr,ur)[]|{id_Y}  \ar[d]^{g} \\
	W \ar@(dl,ul)[]|{id_W} & Z \ar@(dr,ur)[]|{id_Z} \ar[l]_{h}
}
\xymatrix @R=0.4in @C=0.8in{
	FX \ar@(dl,ul)[]|{id_{FX}} \ar[r]^{Ff} \ar[d]_{Fh \circ Fg \circ Ff} \ar[dr]_{Fg \circ Ff} & FY \ar@(dr,ur)[]|{id_{FY}}  \ar[d]^{Fg} \\
	FW \ar@(dl,ul)[]|{id_{FW}} & FZ \ar@(dr,ur)[]|{id_{FZ}} \ar[l]_{Fh}
}
\end{equation*}

Then these individual diagrams are combined in a single diagram to give a schematic representation of $(F,\mathbf{C},\mathbf{D})$ as 
\begin{equation*}
\begin{tikzcd}[row sep=large, column sep=30ex]
(X,FX) \arrow[loop left]{l}{\mathrm{id}_{(X,FX)}} \arrow{dr}{(g\circ f,Fg\cdot Ff)} \arrow{r}{(f,Ff)} \arrow[swap]{d}{(h\circ g\circ f,Fh\cdot Fg\cdot Ff)} & (Y,FY) \arrow{d}{(g,Fg)} \\
(W,FW) & (Z,FZ) \arrow{l}{(h,Fh)}
\end{tikzcd}
\end{equation*}

This definition should not be entirely surprising if we ponder upon the fact that every structure preserving arrow will transport a structure into a similar structure. Then pairing the structure with its image to form a new entity will naturally satisfy axioms of that structure. To us it appears to be more appropriate to call it the \textbf{graph of a functor}. However, since there already exists a notion of a graph of a functor in the category theory with this terminology the reader must exercise caution ; see Remarks~\ref{sec:graph}. Thus we have seen this category is extremely simple to construct using the definition of a functor as a structure preserving morphism. The structure-preserving property of the morphism (functor) plays an extremely crucial role in imparting this structure (category) to the entity $F:\mathbf{C}\rightarrow \mathbf{D}$. Although this construction might not yield any new data apart from the already well-studied functor, it does offer a distinct graph perspective to look at a functor which also turns out to be a fibred category utilized in formulating category-theoretic definition of redundancy in signals as briefly motivated in Section~\ref{motiv} and studied in detail in~\cite{salilp3}. 

Next we prove that this category is a fibration. We begin with recalling the definition of fibration and opfibration.
\begin{definition}[\cite{BJ}]
	Let $P:\mathbf{E} \rightarrow \mathbf{B}$ be a usual functor;
	\begin{itemize}
		
		\item A morphism $f:X \rightarrow Y$ in $\mathbf{E}$ is Cartesian over $u:I \rightarrow J$ in $\mathbf{B}$ if $Pf=u$ and every $g:Z \rightarrow Y$ in $\mathbf{E}$ for which one has $Pg=uw$ for some $w:PZ \rightarrow I$, uniquely determines an $h:Z \rightarrow X$ in $\mathbf{E}$ above $w$ with $fh=g$.
		
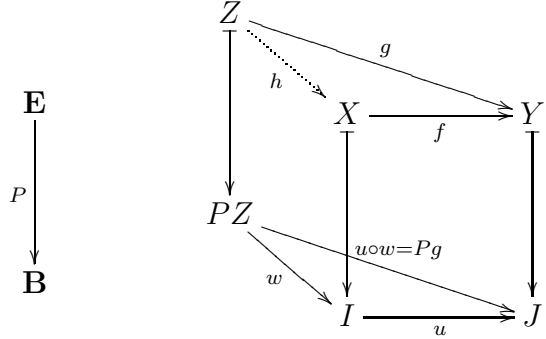
		\begin{equation*}
		\xymatrix{ & \\
			\mathbf{E} \ar[dd]_{P} \\ 
			& \\
			\mathbf{B}}
		\qquad
		\xymatrix{ Z \ar@{|->}[dd] \ar@{.>}[rd]_{h} \ar[rdrr]^{g} & & \\
			& X \ar@{|->}[dd] \ar[rr]_{f} & & Y \ar@{|->}[dd] \\
			PZ \ar[rd]_{w} \ar[rdrr]^{u \circ w =Pg}  & & \\
			& I \ar[rr]_{u} & & J }
		\end{equation*}
		
		\item The functor $P:\mathbf{E} \rightarrow \mathbf{B}$ is a fibration (or a fibred category) if for every $Y \in \mathbf{E}$ and $u:I \rightarrow PY$ in $\mathbf{B}$, there is a Cartesian morphism $f:X \rightarrow Y$ in $\mathbf{E}$ above $u$.
	\end{itemize}
\end{definition}

\begin{definition}[\cite{BJ}]
	Let $P:\mathbf{E} \rightarrow \mathbf{B}$ be a usual functor;
	\begin{itemize}
		
		\item A morphism $f:X \rightarrow Y$ in $\mathbf{E}$ is opCartesian over $u:I \rightarrow J$ in $\mathbf{B}$ if $Pf=u$ and every $g:X \rightarrow Z$ in $\mathbf{E}$ for which one has $Pg=wu$ for some $w:J \rightarrow PZ$, uniquely determines an $h:Y \rightarrow Z$ in $\mathbf{E}$ above $w$ with $hf=w$.
		
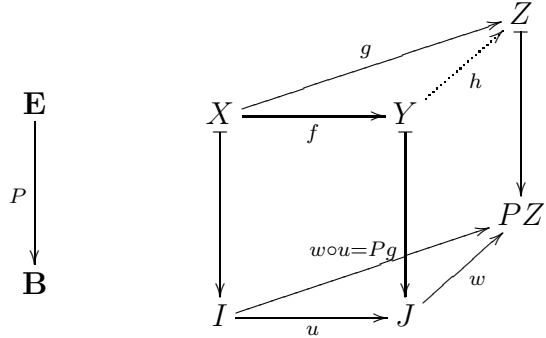
		\begin{equation*}
		\xymatrix{ & \\
			\mathbf{E} \ar[dd]_{P} \\ 
			& \\
			\mathbf{B}}
		\qquad
		\xymatrix{ & & &  Z \ar@{|->}[dd]  \\
			X \ar@{|->}[dd] \ar[rurr]^{g} \ar[rr]_{f} & & Y \ar@{|->}[dd]  \ar@{.>}[ru]_{h} \\
			& & & PZ   \\
			I \ar[rr]_{u} \ar[rurr]^{w \circ u =Pg} & & J \ar[ru]_{w} }
		\end{equation*}
		
		\item The functor $P:\mathbf{E} \rightarrow \mathbf{B}$ is a opfibration (or a opfibred category) if for every $X \in \mathbf{E}$ and $u:PX \rightarrow J$ in $\mathbf{B}$, there is a opCartesian morphism $f:X \rightarrow Y$ in $\mathbf{E}$ above $u$.
	\end{itemize}
\end{definition}

\begin{proposition}
	\label{prop2}	
	Let $P:(F,\mathbf{C},\mathbf{D}) \rightarrow \mathbf{C}$ be the usual first projection functor. Then $P$ is also a split fibration and a split opfibration.
	
\end{proposition}
\begin{proof}
	Let $P:(F,\mathbf{C},\mathbf{D}) \rightarrow \mathbf{C}$ be the first projection functor where $P(X,FX) := X$ for all objects $(X,FX) \in Ob((F,\mathbf{C},\mathbf{D}))$ and $P(f,Ff) := f$ for all morphisms $(f,Ff) \in mor((F,\mathbf{C},\mathbf{D}))$. 
	
	The morphism $(f,Ff):(X,FX) \rightarrow (Y,FY)$ in $(F,\mathbf{C},\mathbf{D})$ is Cartesian over $f:X \rightarrow Y$ in $\mathbf{C}$ since $P(f,Ff)=f$ and every $(g,Fg):(Z,FZ) \rightarrow (Y,FY)$ in $(F,\mathbf{C},\mathbf{D})$ for which we  have $P(g,Fg)=fh$ for some $h:P(Z,FZ) \rightarrow X$, uniquely determines an $(h,Fh):(Z,FZ) \rightarrow (X,FX)$ in $(F,\mathbf{C},\mathbf{D})$ above $h$ with $(f,Ff)\circ(h,Fh)=(g,Fg)$.
	
	
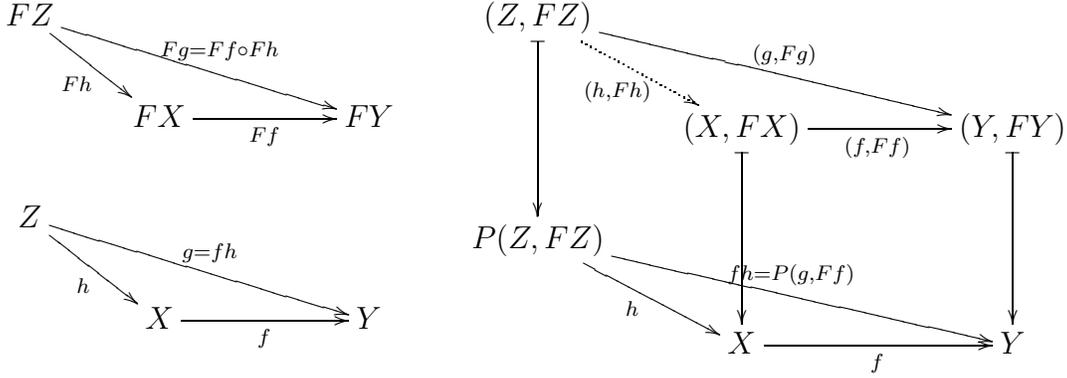
\begin{equation*}
	\xymatrix{ FZ \ar[rd]_{Fh} \ar[rdrr]^{Fg = Ff \circ Fh} & & \\
		& FX \ar[rr]_{Ff} & & FY     \\
		Z \ar[rd]_{h} \ar[rdrr]^{g = fh} & & \\
		& X \ar[rr]_{f} & & Y } 
	\qquad
	\xymatrix{ (Z,FZ) \ar@{|->}[dd] \ar@{.>}[rd]_{(h,Fh)} \ar[rdrr]^{(g,Fg)} & & \\
		& (X,FX) \ar@{|->}[dd] \ar[rr]_{(f,Ff)} & & (Y,FY) \ar@{|->}[dd] \\
		P(Z,FZ) \ar[rd]_{h} \ar[rdrr]^{fh = P(g,Ff)}  & & \\
		& X \ar[rr]_{f} & & Y }
	\end{equation*}
	
	Existence is guaranteed since the arrow $(g,Fg)$ above $fh = g$ is in one-to-one correspondence with arrow $Fg$ on the left the way we have defined our arrows in $(F,\mathbf{C},\mathbf{D})$. Thus corresponding to a well defined $Fh$ we have $(h,Fh)$ above ${h}$. Uniqueness is also guaranteed as there is no other arrow above ${h}$ because $Fh$ is the only image arrow of $h$ guaranteed by the definition of 1-functor $F$ proving that the arrow above ${f}$ is indeed Cartesian. This argument holds for all the arrows of $(F,\mathbf{C},\mathbf{D})$. Thus every arrow in the constructed category acts as a Cartesian lift for the underlying arrow.
	
	Thus functor $P:(F,\mathbf{C},\mathbf{D}) \rightarrow \mathbf{C}$ is indeed a fibration (or a fibred category) since for every $(Y,FY) \in (F,\mathbf{C},\mathbf{D})$ and $f:X \rightarrow P(Y,FY)$ in $\mathbf{C}$, there is a Cartesian morphism $(f,Ff):(X,FX) \rightarrow (Y,FY)$ in $(F,\mathbf{C},\mathbf{D})$ above $f$. It is split since the cleavage satisfies the standard splitting conditions $\gamma(\mathrm{id}_{Y},(Y,FY)) = \mathrm{id}_{(Y,FY)}$ and $\gamma(g,(Z,FZ)) \circ \gamma(f,(Y,FY)) = \gamma(g \circ f,(Z,FZ))$ both of which are a consequence of the fact that ${F}$ is not pseudo but rather a strict functor. In a similar fashion it can be proved that $P:(F,\mathbf{C},\mathbf{D}) \rightarrow \mathbf{C}$ is also a split opfibration using the fact that for every $(X,FX) \in (F,\mathbf{C},\mathbf{D})$ and $f:P(Y,FY) \rightarrow Y$ in $\mathbf{C}$, there is a opCartesian morphism $(f,Ff):(X,FX) \rightarrow (Y,FY)$ in $(F,\mathbf{C},\mathbf{D})$ above $f$.
	
\end{proof}

Thus when $(F,\mathbf{C},\mathbf{D})$ is viewed as fibred on $\mathbf{C}$ it appears as shown in Figure~\ref{fig:1fib}.
\begin{figure}[!h]
	\begin{equation*}
	\xymatrix{
		(X,FX) \ar@/^1pc/[rrrr]^{(g \circ f,FgFf)} \ar@{>}[rr]_{(f,Ff)} &  & (Y,FY) \ar@{>}[rr]_{(g,Fg)}  & & (Z,FZ) \\
		(X,FX) \ar[u]^{(id_X,id_{FX})} \ar@{>}[rr]^{(f,Ff)}  & & (Y,FY) \ar[u]^{(id_Y,id_{FY})} \ar@{>}[rr]^{(g,Fg)} & & (Z,FZ) \ar[u]_{(id_Z,id_{FZ})} \\
		X \ar@/^1pc/[rrrr]^{g \circ f} \ar@(dl,ul)[]|{id_X} \ar[rr]_{f} &  & Y \ar@(dr,ur)[]|{id_Y} \ar[rr]_{g} & &  Z \ar@(dr,ur)[]|{id_Z}
	}
	\end{equation*}
	\caption{$(F,\mathbf{C},\mathbf{D})$ with trivial categories as fibres, on $\mathbf{C}$}
	\label{fig:1fib}
\end{figure}
Once we make the choice of pullbacks (note in this case the pullbacks are not just unique up-to vertical isomorphism but actually unique as there is single object in each fibre) we can observe that indeed $P:(F,\mathbf{C},\mathbf{D}) \rightarrow \mathbf{C}$ becomes a split fibration.

Here every map $f:X \rightarrow Y$ in $\mathbf{C}$ determines a functor $f^*$ in a reverse direction from the whole fibre on $Y$ to the whole fibre on $X$. Such a functor in fibration theory is referred to as \textbf{change-of-base} or \textbf{pullback} functor.
\begin{figure}[!h]
	\begin{equation*}
	\xymatrix@R=0.4in @C=1.2in{
		f^*(Y,FY) = (X,FX) \ar@{.>}[r]^{\bar{f}(Y,FY)=(f,Ff)} \ar@{-->}[d]_{f^*(id_{(Y,FY)})} & (Y,FY) \ar@{->}[d]^{id_{(Y,FY)}} \\
		f^*(Y,FY) = (X,FX) \ar@{.>}[r]_{\bar{f}(Y,FY)=(f,Ff)} & (Y,FY) \\
		X \ar[r]^{f} & Y
	}
	\end{equation*}
	\caption{{Pullback functor $f^*$ corresponding to $f$ in fibration $P:(F,\mathbf{C},\mathbf{D}) \rightarrow \mathbf{C}$}}
	\label{fig:pullfunc}
\end{figure} 
This fibration viewpoint crucially implies thinking of ordinary objects along with their identity as trivial categories which is explained in Section~\ref{act}.

However the Definition~\ref{defn1} of abstract graph of a functor motivates us to characterize a related functor $\mathbb{F}: \mathbf{C} \rightarrow \mathbf{D} \rightarrow \mathbf{Set}$ as its concrete version utilizing 2-category nature of $\mathbf{Set}$.

\begin{definition} [Concrete graph of a functor]
	\label{defn2}
	Consider a (covariant) functor $F:\mathbf{C}\rightarrow \mathbf{D}$ with $(\mathbf{D}, U)$ being a concrete category over $\mathbf{Set}$ or a faithful $U:\mathbf{D} \rightarrow \mathbf{Set}$ with $\mathbb{F} = U \circ F$. Then $(\mathbb{F},\mathbf{C},\mathbf{Set})$ is a category consisting of:
	\begin{itemize}
		\item \textbf{objects}: the pairs $(X,x)$ where $X \in \mbox{Ob}(\mathbf{C})$ and $x \in {\mathbb{F}}X = U({F}X)$
		
		\item \textbf{morphisms}: pairs $(f,\mathbb{F}f|_{x}): (X,x) \rightarrow (Y,y)$ where $f:X \rightarrow Y \in \mathbf{C}$, $y = {\mathbb{F}}f(x)$
		
		\item \textbf{identity}: for $(X,x)$, the morphism $\mathrm{id}_{(X,x)} = (\mathrm{id}_{X},\mathrm{id}_{\mathbb{F}X}|_{x})$
		
		\item \textbf{composition}: $(g,\mathbb{F}g|_{y}) \bullet (f,\mathbb{F}f|_{x}) = (g \circ f,\mathbb{F}(g \circ f)|_{x})$
		
		\item \textbf{unit laws}: for $(f,\mathbb{F}f|_{x})$, 
		$(\mathrm{id}_{Y},\mathrm{id}_{\mathbb{F}Y}|_{y}) \bullet (f,\mathbb{F}f|_{x}) = (f,\mathbb{F}f|_{x}) = (f,\mathbb{F}f|_{x}) \bullet (\mathrm{id}_{X},\mathrm{id}_{\mathbb{F}X}|_{x})$    
		
		\item \textbf{associativity}:  $(h,\mathbb{F}h|_{z})\bullet ((g,\mathbb{F}f|_{y})\bullet (f,\mathbb{F}f|_{x}))=((h,\mathbb{F}h|_{z})\bullet (g,\mathbb{F}f|_{y}))\bullet (f,\mathbb{F}f|_{x}) =(h,\mathbb{F}h|_{z})\bullet (g,\mathbb{F}f|_{y})\bullet (f,\mathbb{F}f|_{x})$
	\end{itemize}
\end{definition}

This category will now be proved only as opfibration in contrast to the dual nature of $P:(F,\mathbf{C},\mathbf{D}) \rightarrow \mathbf{C}$ as both fibration and opfibration. 

\begin{proposition}
	\label{prop3}	
	Let $P:(\mathbb{F},\mathbf{C},\mathbf{Set}) \rightarrow \mathbf{C}$ be the usual first projection functor. Then $P$ is also an split opfibration.
	
\end{proposition}
\begin{proof}
	Let $P:(\mathbb{F},\mathbf{C},\mathbf{Set}) \rightarrow \mathbf{C}$ be the first projection functor where $P(X,x) := X$ for all objects $(X,x) \in Ob((\mathbb{F},\mathbf{C},\mathbf{Set}))$ and $P(f,\mathbb{F}f|_{x}) := f$ for all morphisms $(f,\mathbb{F}f|_{x}) \in mor((\mathbb{F},\mathbf{C},\mathbf{Set}))$. 
	
	The morphism $(f,\mathbb{F}f|_{x}):(X,x) \rightarrow (Y,y)$ in $(\mathbb{F},\mathbf{C},\mathbf{Set})$ is opCartesian over $f:X \rightarrow Y$ in $\mathbf{C}$ since $P(f,\mathbb{F}f|_{x})=f$ and every $(g,\mathbb{F}g|_{x}):(X,x) \rightarrow (Z,z)$ in $(\mathbb{F},\mathbf{C},\mathbf{Set})$ for which we  have $P(g,\mathbb{F}g|_{x})=hf$ for some $h:Y \rightarrow P(Z,z)$, uniquely determines an $(h,\mathbb{F}h|_{y}):(Y,y) \rightarrow (Z,z)$ in $(\mathbb{F},\mathbf{C},\mathbf{Set})$ above $h$ with $(h,\mathbb{F}f|_{y})\circ(f,\mathbb{F}f|_{x})=(g,\mathbb{F}g|_{x})$.
	
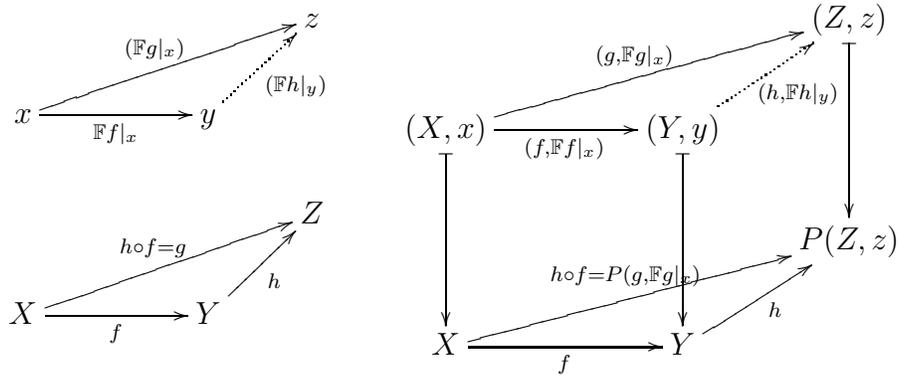
	\begin{equation*}
	\xymatrix{ & & & z \\
		x \ar[rr]_{\mathbb{F}f|_{x}} \ar[rurr]^{(\mathbb{F}g|_{x})} & & y \ar@{.>}[ru]_{(\mathbb{F}h|_{y})}    \\
		& & & Z \\
		X \ar[rr]_{f} \ar[rurr]^{h \circ f =g} & & Y \ar[ru]_{h} } 
	\qquad
	\xymatrix{ & & &  (Z,z) \ar@{|->}[dd]  \\
		(X,x) \ar@{|->}[dd] \ar[rurr]^{(g,\mathbb{F}g|_{x})} \ar[rr]_{(f,\mathbb{F}f|_{x})} & & (Y,y) \ar@{|->}[dd]  \ar@{.>}[ru]_{(h,\mathbb{F}h|_{y})} \\
		& & & P(Z,z)   \\
		X \ar[rr]_{f} \ar[rurr]^{h \circ f =P(g,\mathbb{F}g|_{x})} & & Y \ar[ru]_{h} }
	\end{equation*}
	
	The existence is guaranteed since the arrow $(g,\mathbb{F}g|_{x})$ above $hf = g$ is in one-to-one correspondence with arrow $\mathbb{F}g|_{x}$ on the left where $\mathbb{F}g|_{x} = \mathbb{F}(h \circ f)|_{x}$ using the fact that $\mathbb{F}$ is a usual functor. Thus corresponding to a well defined $\mathbb{F}h|_{y}$ we have $(h,\mathbb{F}h|_{y})$ above ${h}$. The uniqueness is guaranteed as there is no other arrow above ${h}$ with domain as $(Y,y)$ since $\mathbb{F}h|_{y}$ is the unique restriction of function or functor $\mathbb{F}h$ to element or object $y$, proving that the arrow above ${f}$ is indeed opCartesian. This argument holds for all the arrows of $(\mathbb{F},\mathbf{C},\mathbf{Set})$. Thus every arrow in the constructed category acts as a opCartesian lift for the underlying arrow.
	
	Thus functor $P:(\mathbb{F},\mathbf{C},\mathbf{Set}) \rightarrow \mathbf{C}$ is indeed an opfibration since for every $(X,x) \in (\mathbb{F},\mathbf{C},\mathbf{Set})$ and $f:P(X,x) \rightarrow Y$ in $\mathbf{C}$, there is a opCartesian morphism $(f,\mathbb{F}f|_{x}):(X,x) \rightarrow (Y,y)$ in $(\mathbb{F},\mathbf{C},\mathbf{Set})$ above $f$. It is split since the opcleavage satisfies the splitting conditions $\kappa(\mathrm{id}_{X},(X,x)) = \mathrm{id}_{(X,x)}$ and $\kappa(g,(Y,y)) \circ \kappa(f,(X,x)) = \kappa(g \circ f),(X,x))$ both of which are a consequence of the fact that $\mathbb{F}$ is not just pseudo but a strict 2-functor where $\mathbf{Set}$ is treated as a 2-category.  
	
\end{proof}	

\subsection{Remarks on relationship of $(F,\mathbf{C},\mathbf{D})$ to Graph(F)}
\label{sec:graph}
As mentioned earlier, there is an existing notion of graph of a functor \cite{graph} within category theory denoted as $Graph(F)$ which might cause confusion to the reader. The existing graph of a functor is defined using the notions of profunctor and subobject classifier in a topos theoretic way. For a succinct and intuitive overview of topos theory; see \cite{leinstertopos} and references therein. More precisely, the graph of $F$ is the fibration $Graph(F) \rightarrow  \mathbf{C}^{op} \times \mathbf{D}$ classified by $\chi_F$. Using this definition \cite{graph} one can easily recover the ordinary notion of graph of a function as a subset of set $X \times Y$; however the notion of graph of any structure preserving map is itself an object with that structure is not recoverable by this notion since like $\mathbf{Set}$ every category need not have a sub-object classifier. Indeed $(F,\mathbf{C},\mathbf{D})$ as a category the way we defined cannot be recovered since $\mathbf{Cat}$ is not a topos. Hence if the reader wishes to interpret $(F,\mathbf{C},\mathbf{D})$ as the graph of a functor then appropriate care needs to be taken.

\section{Transformation categories: $\int_{\mathbf{C}^{op}} \bar{\mathbf{F}}$,$\int_{\mathbf{C}^{op}} \bar{\mathbb{F}}$, $\mathcal{X} \rtimes_{\mathbf{F}} \mathbf{C}$ , $\mathcal{X} \rtimes_{\mathbb{F}} \mathbf{C}$.}
\label{act}
In previous Section~\ref{fcd} we saw that $(F,\mathbf{C},\mathbf{D})$ was constructed as a graph of a functor and shown as fibred category. However this motivates us to seek a corresponding contravariant functor $\Psi:\mathbf{C} \rightarrow \mathbf{Cat}$ thought of as covariant functor $\bar{\Psi}:\mathbf{C}^{op} \rightarrow \mathbf{Cat}$ whose Grothendieck construction must yield a fibred category at least equivalent to $(F,\mathbf{C},\mathbf{D})$ or possibly its opposite. This corresponding functor as it turns out is precisely the contravariant functor $\bar{\mathbf{F}}: \mathbf{C}^{op} \rightarrow \mathbf{Cat}$ which is constructed from $\bar{F}:\mathbf{C}^{op} \rightarrow \mathbf{D}$ by trivial categorification of $\mathbf{D}$ objects or more precisely by post composing with trivial inclusion functor $I$ to form $\bar{\mathbf{F}} = I \circ \bar{F}$ as explained in this section. 

\subsection{Trivial categorification of $\mathbf{D}$ objects}
\label{trivcat}
To be able do Grothendieck construction on ordinary functor $F: \mathbf{C} \rightarrow \mathbf{D}$ we need a functor $I: \mathbf{D} \rightarrow \mathbf{Cat}$. The recipe for definition of this functor is intuitively suggested by Figure~\ref{fig:pullfunc} using classic pullback interpretation of a fibration. As shown then in Figure~\ref{fig:trivialcat} the object $FX$ is mapped to a category $I(FX)$ which consists of a single object $FX$ with its identity arrow $id_{FX}$. Then every morphism $Ff:FX \rightarrow FY$ of the category $\mathbf{D}$  is naturally mapped to a functor $I(Ff)$. The commutative diagram is the unit law in $\mathbf{D}$. This we term {\bf trivial categorification} in the sense that every object is treated as trivial category while every arrow correspondingly gives rise to a functor. Note that this can be done unambiguously for every object of any arbitrary category since for every object $FX$ the identity arrow ${id_{FX}}$ uniquely exists setting up a definition of trivial category $\mathbf{F}X = I(FX)$. Next for every arrow $Ff: FX \rightarrow FY$ the well-defined domain and codomain sets up a definition for a functor $I(Ff) = \mathbf{F}f: \mathbf{F}X \rightarrow \mathbf{F}Y$  (one can easily verify functor axioms) sending the domain trivial category to codomain trivial category and the corresponding commutative diagram always holds as a unit law in the original category $\mathbf{D}$. Note that $I$ indeed satisfies axioms of a functor.    

\begin{figure}[!h]
	\begin{equation*}
	\xymatrix@R=0.4in @C=0.8in{
		FX \ar@{.>}[r]^{Ff} \ar@{>}[d]_{id_{FX}} & I(Ff)(FX)=FY \ar@{>}[d]^{{I(Ff)(id_{FX}})=id_{FY}} \\
		FX \ar@{.>}[r]_{Ff} & I(Ff)(FX)=FY \\
		FX \ar[r]^{Ff} & FY
	}
	\end{equation*}
	\caption{Objects $FX,FY$ mapped to trivial categories and morphisms $Ff$ mapped to functors by $I$.}
	\label{fig:trivialcat}
\end{figure}

Thus corresponding to $F: \mathbf{C} \rightarrow \mathbf{D}$, we have $\mathbf{F}: \mathbf{C} \rightarrow \mathbf{Cat}$ as a well defined functor with $\mathbf{F} = I \circ F$. The terminology categorification is in line with its well-studied usage~\cite{baezcat} in context of replacing sets with categories and functions with functors, since here we are replacing objects (either abstract or concrete such as Hilbert spaces, smooth manifolds etc.) with categories while morphisms are replaced with functors.

\subsection{Right action induced by a functor}   
\label{rightact}
The basic duality fact in~\cite{CWM} states that every contravariant functor could be written as covariant using the concept of opposite category. This can be utilized in a reverse way to express a covariant functor as a contravariant functor. The contravariant form is essential since the original version of Grothendieck construction~\cite{SGA1} is on a contravariant pseudofunctor and the same version is utilized by~\cite{BJ} in the context of categorical logic. As in~\cite{SGA1} let us denote dual arrows $f^{op}$ by $f^{\circ}$ in diagrams for convenience.

\begin{lemma}\cite{CWM}
	\label{contra} 
	To every (covariant) functor $F: \mathbf{C} \rightarrow \mathbf{D}$ we can always associate a corresponding (contravariant) functor $\bar{F}: \mathbf{C}^{op} \rightarrow \mathbf{D}$.
\end{lemma}
\begin{proof}
	Consider the functor $F: \mathbf{C} \rightarrow \mathbf{D}$. By definition it assigns to each object $X \in \mathbf{C}$ an object $FX \in \mathbf{D}$ and to each arrow $f:X \rightarrow Y \in \mathbf{C}$ an arrow $Ff:FX \rightarrow FY \in \mathbf{D}$ with $F(g \circ f) = (Fg)\bullet(Ff)$ whenever $g \circ f$ is defined. Now we write $\bar{F}f^{op}$ for $Ff$ and $\bar{F}X = {F}X$ ; then one can define a functor $\bar{F}$ which is contravariant from $\mathbf{C}^{op}$ to $\mathbf{D}$ assigning to each object $X \in \mathbf{C}^{op}$ an object $\bar{F}X \in \mathbf{D}$ and to each arrow $f^{op}:Y \rightarrow X$ an arrow $\bar{F}f^{op}:\bar{F}X \rightarrow \bar{F}Y$ (in the opposite direction) all in such a way that $\bar{F}(\mathrm{id}_X) = \mathrm{id}_{\bar{F}X}$ and $\bar{F}(f^{op} \circ g^{op}) = (\bar{F}g^{op})\bullet(\bar{F}f^{op})$ whenever the composite $f^{op} \circ g^{op}$ is defined in $\mathbf{C}^{op}$. Thus the contravariant functor inverts the order of composition as explained in \cite{CWM}.  
	\[
	\begin{gathered}
	\xymatrix{
		FX \ar[r]^{Ff} &
		FY \ar[r]^{Fg}  &
		FZ  \\
		X \ar[r]_{f} \ar[u]_{F} &
		Y \ar[r]_{g} \ar[u]^{F} &
		Z \ar[u]^{F}
	}
	\end{gathered}
	\qquad
	\begin{gathered}
	\xymatrix{
		\bar{F}X=FX \ar[r]^{Ff=\bar{F}f^{\circ}} &
		\bar{F}Y=FY \ar[r]^{Fg=\bar{F}g^{\circ}}  &
		\bar{F}Z=FZ  \\
		X  \ar[u]_{\bar{F}} &
		Y \ar[l]^{f^{\circ}} \ar[u]^{\bar{F}} &
		Z \ar[l]^{g^{\circ}} \ar[u]^{\bar{F}}
	}
	\end{gathered}
	\]
\end{proof}

The functor $\bar{F}: \mathbf{C}^{op} \rightarrow \mathbf{D}$ as we have defined is only symbolic on objects of the original category and in general is undefined on $\mathbf{C}$. Only in the case of categories such as groups, groupoids, partial monic (which are in essence partial groupoids) etc; the opposite category is isomorphic to original category i.e $\mathbf{C}^{op} \cong \mathbf{C}$. In such a case we can meaningfully write a contravariant functor $\bar{F}: \mathbf{C} \rightarrow \mathbf{D}$ and therefore the category $\int_{\mathbf{C}}\bar{\mathbf{F}}$ could be obtained using the Grothendieck construction. More often is the case of equivalence $\mathbf{C}^{op} \cong \mathbf{E}$ where opposite category is equivalent to some other category $\mathbf{E}$ rather than original $\mathbf{C}$.

Lemma~\ref{contra} together with $\mathbf{F} = I \circ F$ immediately suggests Grothendieck construction in case of ordinary contravariant functors (to ordinary category $\mathbf{D}$ instead of $\mathbf{Cat}$). Given an ordinary covariant functor ${F}: \mathbf{C} \rightarrow \mathbf{D}$, we first think of it as a contravariant functor $\bar{F}: \mathbf{C}^{op} \rightarrow \mathbf{D}$. This lets us define Grothendieck construction on it by using $\bar{\mathbf{F}}: \mathbf{C}^{op} \rightarrow \mathbf{Cat}$ or more precisely the category $\int_{\mathbf{C}^{op}} \bar{\mathbf{F}}$ where $\bar{\mathbf{F}} = I \circ \bar{F}$. This is in disguise \textbf{left action} of $\mathbf{C}^{op}$ which through duality can be interpreted as right action of $\mathbf{C}$. However unless $\mathbf{C}$ is isomorphic to its opposite $\mathbf{C}^{op}$ the right action and left action categories are in general not isomorphic; see Section~\ref{dual}.

\begin{definition}[Abstract Right action induced by a functor]
	\label{defn3}
	Consider a strict contravariant functor $\bar{F}: \mathbf{C}^{op} \rightarrow \mathbf{D}$ between small categories thought of as $\bar{\mathbf{F}}: \mathbf{C}^{op} \rightarrow \mathbf{Cat}$ (with $\bar{\mathbf{F}} = I \circ \bar{F}$ and $I:\mathbf{D} \rightarrow \mathbf{Cat}$ as defined). Then Grothendieck construction of $\bar{\mathbf{F}}$ is a category $\int_{\mathbf{C}^{op}} \bar{\mathbf{F}}$ with

	\begin{itemize}
		\item \textbf{objects}: the pairs $(X,\bar{F}X)$ where $X \in \mbox{Ob}(\mathbf{C}^{op})$ and $\bar{F}X \in \mbox{Ob}(\mathbf{D})$
		
		\item \textbf{morphisms}: ${\int_{\mathbf{C}^{op}} \bar{\mathbf{F}}}((Y,\bar{F}Y),(X,\bar{F}X))$ are pairs $(f^{\circ},\mathrm{id}_{\bar{F}Y})$ where $f^{\circ}:Y \rightarrow X \in \mathbf{C}^{op}$  $\mathrm{id}_{\bar{F}Y}:\bar{F}Y \rightarrow \bar{\mathbf{F}}f^{\circ}(\bar{F}X)$
		
		\item \textbf{identity}: for $(X,\bar{F}X)$, the morphism $\mathrm{id}_{(X,\bar{F}X)} = (\mathrm{id}_{X},\mathrm{id}_{\bar{F}X})$
		
		\item \textbf{composition}: $(f^{\circ},\mathrm{id}_{\bar{F}Y}) \bullet (g^{\circ},\mathrm{id}_{\bar{F}Z}) = (f^{\circ} \circ g^{\circ},\bar{\mathbf{F}}g^{\circ}(\mathrm{id}_{\bar{F}Y}) \cdot \mathrm{id}_{\bar{F}Z}) = (f^{\circ} g^{\circ},\mathrm{id}_{\bar{F}Z})$ since $\bar{\mathbf{F}}g^{\circ}(\mathrm{id}_{\bar{F}Y}) :  \bar{\mathbf{F}}g^{\circ}\bar{F}Y \rightarrow \bar{\mathbf{F}}g^{\circ}\bar{\mathbf{F}}f^{\circ}(\bar{F}X)$ 
		
		\item \textbf{unit laws}: for $(f^{\circ},\mathrm{id}_{\bar{F}Y})$, 
		$(\mathrm{id}_{X},\mathrm{id}_{\bar{F}X}) \bullet (f^{\circ},\mathrm{id}_{\bar{F}Y}) = (f^{\circ},\mathrm{id}_{\bar{F}Y}) = (f^{\circ},\mathrm{id}_{\bar{F}Y}) \bullet (\mathrm{id}_{Y},\mathrm{id}_{\bar{F}Y})$    
		
		\item \textbf{associativity}:  $(f^{\circ},\mathrm{id}_{\bar{F}Y})\bullet ((g^{\circ},\mathrm{id}_{\bar{F}Z})\bullet (h^{\circ},\mathrm{id}_{\bar{F}W}))=((f^{\circ},\mathrm{id}_{\bar{F}Y})\bullet (g^{\circ},\mathrm{id}_{\bar{F}Z}))\bullet (h^{\circ},\mathrm{id}_{\bar{F}W}) =(f^{\circ},\mathrm{id}_{\bar{F}Y})\bullet (g^{\circ},\mathrm{id}_{\bar{F}Z})\bullet (h^{\circ},\mathrm{id}_{\bar{F}W})$
	\end{itemize}
\end{definition}
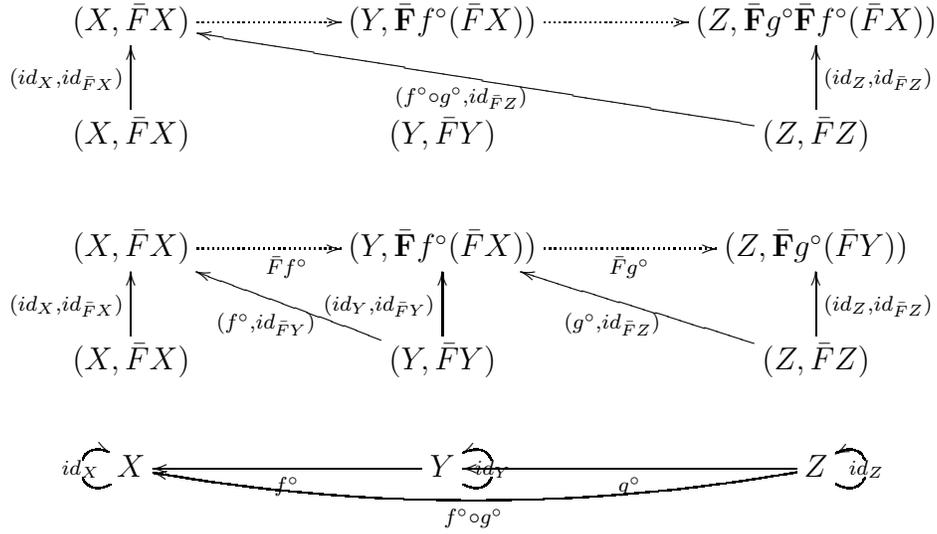
\begin{figure}[!h]
	\begin{equation*}
	\xymatrix{
		(X,\bar{F}X) \ar@{.>}[rr]  &  & (Y,\bar{\mathbf{F}}f^{\circ}(\bar{F}X)) \ar@{.>}[rr]  & & (Z, \bar{\mathbf{F}}g^{\circ}\bar{\mathbf{F}}f^{\circ}(\bar{F}X))  \\
		(X,\bar{F}X) \ar[u]^{(id_X,id_{\bar{F}X})}  & & (Y,\bar{F}Y)  & & (Z,\bar{F}Z) \ar[u]_{(id_Z,id_{\bar{F}Z})} \ar@{>}[llllu]^{(f^{\circ} \circ g^{\circ},id_{\bar{F}Z})} \\
		(X,\bar{F}X) \ar@{.>}[rr]_{\bar{F}f^{\circ}}  &  & (Y,\bar{\mathbf{F}}f^{\circ}(\bar{F}X)) \ar@{.>}[rr]_{\bar{F}g^{\circ}}    & & (Z,\bar{\mathbf{F}}g^{\circ}(\bar{F}Y)) \\
		(X,\bar{F}X) \ar[u]^{(id_X,id_{\bar{F}X})}   & & (Y,\bar{F}Y) \ar[u]^{(id_Y,id_{\bar{F}Y})}  \ar@{>}[llu]^{(f^{\circ},id_{\bar{F}Y})} & & (Z,\bar{F}Z) \ar[u]_{(id_Z,id_{\bar{F}Z})} \ar@{>}[llu]^{(g^{\circ},id_{\bar{F}Z})} \\
		X  \ar@(dl,ul)[]|{id_X}  &  & Y \ar@(dr,ur)[]|{id_Y} \ar[ll]^{{f}^{\circ}} & &  Z \ar@(dr,ur)[]|{id_Z} \ar[ll]^{{g}^{\circ}} \ar@/^1pc/[llll]^{{f}^{\circ} \circ {g}^{\circ}}
	}
	\end{equation*}   
	\caption{$\int_{\mathbf{C}^{op}}\bar{\mathbf{F}}$ fibred on $\mathbf{C}^{op}$; dotted arrows show $\mathbf{D}$ morphisms as actions}
	\label{fig:13fib}
\end{figure}

The abstract right action immediately motivates us to define a concrete version given there is some underlying category of $\mathbf{D}$. We define this for $\mathbf{Set}$ (or a construct) however the case for any underlying category $\mathbf{X}$ should not be more difficult.    

\begin{definition}[Concrete Right action induced by a functor]
	\label{defn4}
	Consider a strict contravariant functor $\bar{F}: \mathbf{C}^{op} \rightarrow \mathbf{D}$ between small categories with $(\mathbf{D}, U)$ being a concrete category over $\mathbf{Set}$ or a faithful functor $U:\mathbf{D} \rightarrow \mathbf{Set}$. Then ${\bar{\mathbb{F}}} = U \circ \bar{F}$ and $\int_{\mathbf{C}^{op}} \bar{\mathbb{F}}$ is a category with

	\begin{itemize}
		\item \textbf{objects}: the pairs $(X,x)$ where $X \in \mbox{Ob}(\mathbf{C}^{op})$ and $x \in \bar{\mathbb{F}}X = U(\bar{F}X)$
		
		\item \textbf{morphisms}: pairs $(f^{\circ},y): (Y,y) \rightarrow (X,x)$ where $f^{\circ}:Y \rightarrow X \in \mathbf{C}^{op}$, $x = \bar{\mathbb{F}}f(y)$
		
		\item \textbf{identity}: for $(X,x)$, the morphism $\mathrm{id}_{(X,x)} = (\mathrm{id}_{X},x)$
		
		\item \textbf{composition}: $(f^{\circ},y) \bullet (g^{\circ},z) = (f^{\circ} \circ g^{\circ},{z})$ since $z = \bar{\mathbf{F}}g^{\circ}\bar{\mathbf{F}}f^{\circ}(x)$
		
		\item \textbf{unit laws}: for $(f^{\circ},y)$, 
		$(\mathrm{id}_{X},{x}) \bullet (f^{\circ},y) = (f^{\circ},y) = (f^{\circ},y) \bullet (\mathrm{id}_{Y},{y})$    
		
		\item \textbf{associativity}:  $(h^{\circ},{x})\bullet ((f^{\circ},{y})\bullet (g^{\circ},{z}))=((h^{\circ},{x})\bullet (f^{\circ},{y}))\bullet (g^{\circ},{z}) =(h^{\circ},{x})\bullet (f^{\circ},{y})\bullet (g^{\circ},{z})$
	\end{itemize}
\end{definition}

\begin{figure}[!h]
	\begin{equation*}
	\xymatrix{
		(X,x) \ar@{.>}[rr]  &  & (Y,\bar{\mathbf{F}}f^{\circ}(x)) \ar@{.>}[rr]  & & (Z, \bar{\mathbf{F}}g^{\circ}\bar{\mathbf{F}}f^{\circ}(x))  \\
		(X,x) \ar[u]^{(id_X,x)}  & & (Y,\bar{F}Y)  & & (Z,\bar{F}Z) \ar[u]_{(id_Z,z)} \ar@{>}[llllu]^{(f^{\circ} \circ g^{\circ},z)} \\
		(X,x) \ar@{.>}[rr]_{U({\bar{F}f^{\circ}})}  &  & (Y,\bar{\mathbf{F}}f^{\circ}(\bar{F}X)) \ar@{.>}[rr]_{U({\bar{F}g^{\circ}})}    & & (Z,\bar{\mathbf{F}}g^{\circ}(\bar{F}Y)) \\
		(X,x) \ar[u]^{(id_X,x)}   & & (Y,y) \ar[u]^{(id_Y,y)}  \ar@{>}[llu]^{(f^{\circ},y)} & & (Z,z) \ar[u]_{(id_Z,z)} \ar@{>}[llu]^{(g^{\circ},z)} \\
		X  \ar@(dl,ul)[]|{id_X}  &  & Y \ar@(dr,ur)[]|{id_Y} \ar[ll]^{{f}^{\circ}} & &  Z \ar@(dr,ur)[]|{id_Z} \ar[ll]^{{g}^{\circ}} \ar@/^1pc/[llll]^{{f}^{\circ} \circ {g}^{\circ}}
	}
	\end{equation*}   
	\caption{$\int_{\mathbf{C}}\bar{\mathbb{F}}$ fibred on $\mathbf{C}^{op}$; dotted arrows denote concrete functions as right actions}
	\label{fig:rightconcretefib}
\end{figure}
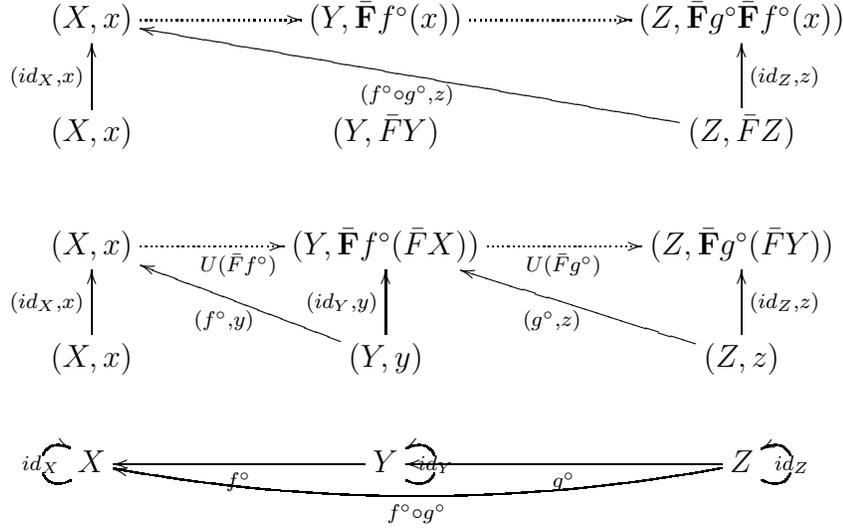

In a dual sense we define the abstract and concrete versions of left actions induced by a functor and then proceed to describe this duality in Section~\ref{dual}.

\subsection{Left action induced by a functor}
\label{leftact}

\begin{definition}[Abstract Left action induced by a functor]
	\label{defn5}
	Consider a (covariant) functor $F: \mathbf{C} \rightarrow \mathbf{D}$ between small categories thought of as covariant $\mathbf{F}: \mathbf{C} \rightarrow \mathbf{Cat}$ (with $\mathbf{F} = I \circ {F}$ and $I:\mathbf{D} \rightarrow \mathbf{Cat}$ as defined). Then $\mathcal{X} \rtimes_{\mathbf{F}} \mathbf{C}$ (or $(\int_{\mathbf{C}^{op}} \bar{\mathbf{F}})^{op}$) is a category with 
	
	\begin{itemize}
		\item \textbf{objects}: the pairs $(X,FX)$ where $X \in \mbox{Ob}(\mathbf{C})$ and $FX \in \mbox{Ob}(\mathbf{D})$
		
		\item \textbf{morphisms}: $\mathcal{X} \rtimes_{\mathbf{F}} \mathbf{C}((X,FX),(Y,FY))$ are pairs $(f,\mathrm{id}_{FY})$ where $f:X \rightarrow Y \in \mathbf{C}$, $\mathrm{id}_{FY}:\mathbf{F}f(FX) \rightarrow FY$
		
		\item \textbf{identity}: for $(X,FX)$, the morphism $\mathrm{id}_{(X,FX)} = (\mathrm{id}_{X},\mathrm{id}_{FX})$
		
		\item \textbf{composition}: $(g,\mathrm{id}_{FZ}) \bullet (f,\mathrm{id}_{FY}) = (g \circ f,\mathrm{id}_{FZ} \cdot \mathbf{F}g(\mathrm{id}_{FY})) = (gf,\mathrm{id}_{FZ})$ since $\mathbf{F}g(\mathrm{id}_{FY}) = \mathbf{F}g\mathbf{F}f(FX) \rightarrow \mathbf{F}g(FY)$ 
		
		\item \textbf{unit laws}: for $(f,\mathrm{id}_{FY})$, 
		$(\mathrm{id}_{Y},\mathrm{id}_{FY}) \bullet (f,\mathrm{id}_{FY}) = (f,\mathrm{id}_{FY}) = (f,\mathrm{id}_{FY}) \bullet (\mathrm{id}_{X},\mathrm{id}_{FX})$    
		
		\item \textbf{associativity}:  $(h,\mathrm{id}_{FW})\bullet ((g,\mathrm{id}_{FZ})\bullet (f,\mathrm{id}_{FY}))=((h,\mathrm{id}_{FW})\bullet (g,\mathrm{id}_{FZ}))\bullet (f,\mathrm{id}_{FY}) =(h,\mathrm{id}_{FW})\bullet (g,\mathrm{id}_{FZ})\bullet (f,\mathrm{id}_{FY})$
	\end{itemize}
	\begin{figure}[!h]
		\begin{equation*}
		\xymatrix{
			(X,FX)   &  & (Y,FY)  & & (Z,FZ) \\
			(X,FX) \ar[u]^{(id_X,id_{FX})} \ar@{.>}[rr] \ar[rrrru]^{{(g \circ f,id_{FZ})}} & & (Y,\mathbf{F}f(FX)) \ar@{.>}[rr]  & & (Z,\mathbf{F}g\mathbf{F}f(FX)) \ar[u]_{(id_Z,id_{FZ})} \\
			(X,FX)   &  & (Y,FY)  & & (Z,FZ) \\
			(X,FX) \ar[u]^{(id_X,id_{FX})} \ar@{.>}[rr]^{Ff}  \ar@{>}[rru]^{(f,id_{FY})} & & (Y,\mathbf{F}f(FX)) \ar[u]^{(id_Y,id_{FY})} \ar@{.>}[rr]^{Fg} \ar@{>}[rru]^{(f,id_{FZ})} & & (Z,\mathbf{F}g(FY)) \ar[u]_{(id_Z,id_{FZ})} \\
			X \ar@/^1pc/[rrrr]^{g \circ f} \ar@(dl,ul)[]|{id_X} \ar[rr]_{f} &  & Y \ar@(dr,ur)[]|{id_Y} \ar[rr]_{g} & &  Z \ar@(dr,ur)[]|{id_Z}
		}
		\end{equation*}
		\caption{$\mathcal{X} \rtimes_{\mathbf{F}} \mathbf{C}$ fibred on $\mathbf{C}$; dotted arrows show $\mathbf{D}$ morphisms as actions}
		\label{fig:12fib}
	\end{figure}
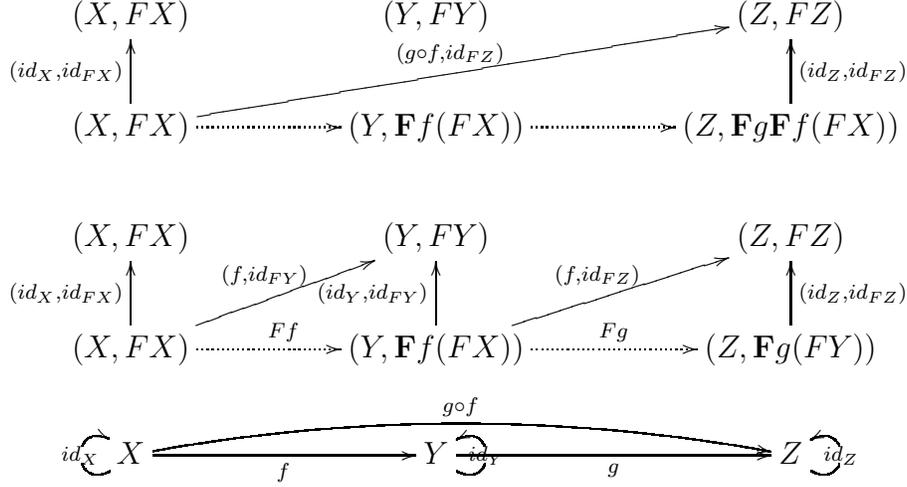 
	
\end{definition} 

The abstract left action again motivates us to define a concrete version given there is some underlying category of $\mathbf{D}$. We define this for $\mathbf{Set}$ (commonly known as construct) however the reader might work out for any underlying category $\mathbf{X}$.    

\begin{definition}[Concrete Left action induced by a functor]
	\label{defn6}
	Consider a covariant functor ${F}: \mathbf{C} \rightarrow \mathbf{D}$ between small categories with $(\mathbf{D}, U)$ being a concrete category over $\mathbf{Set}$ or a faithful $U:\mathbf{D} \rightarrow \mathbf{Set}$. Then ${\mathbb{F}} = U \circ {F}$ and $\mathcal{X} \rtimes_{\mathbb{F}} \mathbf{C}$ (or $(\int_{\mathbf{C}^{op}} \bar{\mathbb{F}})^{op}$) is a category with

	\begin{itemize}
		\item \textbf{objects}: the pairs $(X,x)$ where $X \in \mbox{Ob}(\mathbf{C})$ and $x \in {\mathbb{F}}X = U({F}X)$
		
		\item \textbf{morphisms}: pairs $(f,y): (X,x) \rightarrow (Y,y)$ where $f:X \rightarrow Y \in \mathbf{C}$, $y = {\mathbb{F}}f(x)$
		
		\item \textbf{identity}: for $(X,x)$, the morphism $\mathrm{id}_{(X,x)} = (\mathrm{id}_{X},x)$
		
		\item \textbf{composition}: $(g,{z}) \bullet (f,{y}) = (g \circ f,z \cdot \mathbb{F}g(y)) = (g \circ f,z)$ since $z = {\mathbb{F}}g{\mathbb{F}}f(x)$ 
		
		\item \textbf{unit laws}: for $(f,{y})$, 
		$(\mathrm{id}_{Y},{y}) \bullet (f,{y}) = (f,{y}) = (f,{y}) \bullet (\mathrm{id}_{X},{x})$    
		
		\item \textbf{associativity}:  $(h,{w})\bullet ((g,{z})\bullet (f,{y}))=((h,{w})\bullet (g,{z}))\bullet (f,{y}) =(h,{w})\bullet (g,{z})\bullet (f,{y})$
	\end{itemize}
\end{definition}

\begin{figure}[!h]
	\begin{equation*}
	\xymatrix{
		(X,x)   &  & (Y,y)  & & (Z,z) \\
		(X,x) \ar[u]^{(id_X,x)} \ar@{.>}[rr] \ar[rrrru]^{{(g \circ f,z)}} & & (Y,\mathbb{F}f(x)) \ar@{.>}[rr]  & & (Z,\mathbb{F}g\mathbb{F}f(x)) \ar[u]_{(id_Z,z)} \\
		(X,x)   &  & (Y,y)  & & (Z,z) \\
		(X,x) \ar[u]^{(id_X,x)} \ar@{.>}[rr]^{U(Ff)}  \ar@{>}[rru]^{(f,y)} & & (Y,\mathbb{F}f(x)) \ar[u]^{(id_Y,y)} \ar@{.>}[rr]^{U(Fg)} \ar@{>}[rru]^{(g,z)} & & (Z,\mathbb{F}g(y)) \ar[u]_{(id_Z,z)} \\
		X \ar@/^1pc/[rrrr]^{g \circ f} \ar@(dl,ul)[]|{id_X} \ar[rr]_{f} &  & Y \ar@(dr,ur)[]|{id_Y} \ar[rr]_{g} & &  Z \ar@(dr,ur)[]|{id_Z}
	}
	\end{equation*}
	\caption{$\mathcal{X} \rtimes_{\mathbb{F}} \mathbf{C}$ fibred on $\mathbf{C}$; dotted arrows denote concrete functions as left actions}
	\label{fig:leftconcretefib}
\end{figure}
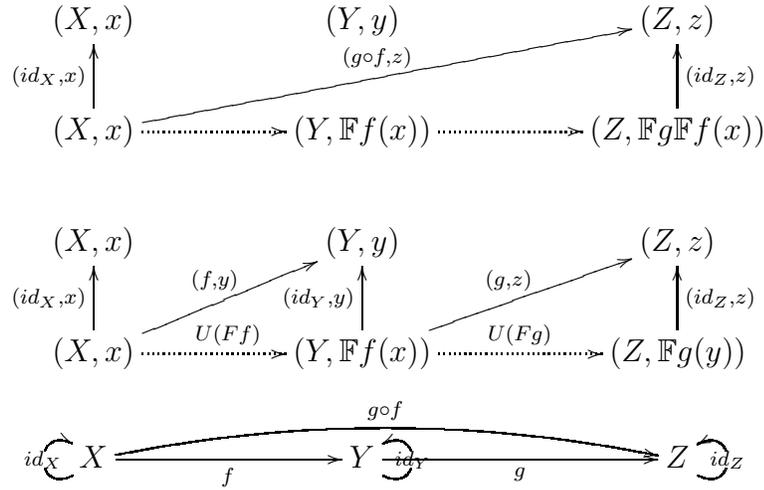 

\subsection{Duality between the categories defined as right actions and left actions}
\label{dual}

Now observing carefully we can discern that the opposite of $\int_{\mathbf{C}^{op}} \bar {\mathbf{F}}$ is identical to $\mathcal{X} \rtimes_{\mathbf{F}} \mathbf{C}$ or precisely $(\int_{\mathbf{C}^{op}} \bar{\mathbf{F}})^{op} = \mathcal{X} \rtimes_{\mathbf{F}} \mathbf{C}$. For this note that we first make use of the fact that $\bar{F}X = FX$ and then $(\mathbf{C}^{op})^{op} = \mathbf{C}$. Thus objects in  $(\int_{\mathbf{C}^{op}} \bar{\mathbf{F}})^{op}$ are the pairs $(X,\bar{F}X)$ which are same as $(X,FX)$ of $\mathcal{X} \rtimes_{\mathbf{F}} \mathbf{C}$. Next consider the arrow $(f^{\circ},\mathrm{id}_{\bar{F}Y}): (Y,\bar{F}Y) \rightarrow (X,\bar{F}X)$ of $\int_{\mathbf{C}^{op}} \bar{\mathbf{F}}$. The opposite arrow of this in $(\int_{\mathbf{C}^{op}} \bar{\mathbf{F}})^{op}$ is $(f,\mathrm{id}_{{F}Y}): (X,{F}X) \rightarrow (Y,{F}Y)$ which is same as the arrow of $\mathcal{X} \rtimes_{\mathbf{F}} \mathbf{C}$. Indeed the reader can verify that these categories are identical.

In the special case of groups, groupoids, partial groupoids (partial monic) categories where arrows can be inverted uniquely we have ${\mathbf{C}^{op}} \cong {\mathbf{C}}$ via the inverse map on arrows and therefore both these conventions of right and left actions are the same. In these cases the abstract and concrete actions reduce to $\int_{\mathbf{C}} \bar{\mathbf{F}}$ and  $\int_{\mathbf{C}} \bar{\mathbb{F}}$ respectively as defined next. These are isomorphic to left action categories.

\begin{definition}[Abstract Right action induced by a functor]
	\label{defn7}
	Consider a strict contravariant functor $\bar{F}: \mathbf{C} \rightarrow \mathbf{D}$ between small categories thought of as $\bar{\mathbf{F}}: \mathbf{C} \rightarrow \mathbf{Cat}$ (with $\bar{\mathbf{F}} = I \circ \bar{F}$ and $I:\mathbf{D} \rightarrow \mathbf{Cat}$ as defined). Then $\int_{\mathbf{C}} \bar{\mathbf{F}}$ is a category with

	\begin{itemize}
		\item \textbf{objects}: the pairs $(X,\bar{F}X)$ where $X \in \mbox{Ob}(\mathbf{C})$ and $\bar{F}X \in \mbox{Ob}(\mathbf{D})$
		
		\item \textbf{morphisms}: ${\int_{\mathbf{C}} \bar{\mathbf{F}}}((X,\bar{F}X),(Y,\bar{F}Y))$ are pairs $(f,\mathrm{id}_{\bar{F}X})$ where $f:X \rightarrow Y \in \mathbf{C}$, $\mathrm{id}_{\bar{F}X}:\bar{F}X \rightarrow \bar{\mathbf{F}}f(\bar{F}Y)$
		
		\item \textbf{identity}: for $(X,\bar{F}X)$, the morphism $\mathrm{id}_{(X,\bar{F}X)} = (\mathrm{id}_{X},\mathrm{id}_{\bar{F}X})$
		
		\item \textbf{composition}: $(g,\mathrm{id}_{\bar{F}Y}) \bullet (f,\mathrm{id}_{\bar{F}X}) = (g \circ f,\bar{\mathbf{F}}f(\mathrm{id}_{\bar{F}Y}) \cdot \mathrm{id}_{\bar{F}X}) = (gf,\mathrm{id}_{\bar{F}X})$ since $\bar{\mathbf{F}}f(\mathrm{id}_{\bar{F}Y}): \bar{\mathbf{F}}f\bar{F}Y \rightarrow \bar{\mathbf{F}}f\bar{\mathbf{F}}g(\bar{F}Z)$ 
		
		\item \textbf{unit laws}: for $(f,\mathrm{id}_{\bar{F}X})$, 
		$(\mathrm{id}_{Y},\mathrm{id}_{\bar{F}Y}) \bullet (f,\mathrm{id}_{\bar{F}X}) = (f,\mathrm{id}_{\bar{F}X}) = (f,\mathrm{id}_{\bar{F}X}) \bullet (\mathrm{id}_{X},\mathrm{id}_{\bar{F}X})$    
		
		\item \textbf{associativity}:  $(h,\mathrm{id}_{\bar{F}Z})\bullet ((g,\mathrm{id}_{\bar{F}Y})\bullet (f,\mathrm{id}_{\bar{F}X}))=((h,\mathrm{id}_{\bar{F}Z})\bullet (g,\mathrm{id}_{\bar{F}Y}))\bullet (f,\mathrm{id}_{\bar{F}X}) =(h,\mathrm{id}_{\bar{F}Z})\bullet (g,\mathrm{id}_{\bar{F}Y})\bullet (f,\mathrm{id}_{\bar{F}X})$
	\end{itemize}
\end{definition}
\begin{figure}[!h]
	\begin{equation*}
	\xymatrix{
		(X,\bar{\mathbf{F}}f\bar{\mathbf{F}}g(\bar{F}Z))  &  & (Y,\bar{\mathbf{F}}g(\bar{F}Z)) \ar@{.>}[ll]  & & (Z,\bar{F}Z) \ar@{.>}[ll] \\
		(X,\bar{F}X) \ar[u]^{(id_X,id_{\bar{F}X})} \ar@{>}[rrrru]^{(g \circ f,id_{\bar{F}X})} & & (Y,\bar{F}Y)  & & (Z,\bar{F}Z) \ar[u]_{(id_Z,id_{\bar{F}Z})} \\
		(X,\bar{\mathbf{F}}f(\bar{F}Y))  &  & (Y,\bar{\mathbf{F}}g(\bar{F}Z)) \ar@{.>}[ll]_{\bar{F}f}  & & (Z,\bar{F}Z) \ar@{.>}[ll]_{\bar{F}g} \\
		(X,\bar{F}X) \ar[u]^{(id_X,id_{\bar{F}X})}  \ar@{>}[rru]^{(f,id_{\bar{F}X})} & & (Y,\bar{F}Y) \ar[u]^{(id_Y,id_{\bar{F}Y})} \ar@{>}[rru]^{(g,id_{\bar{F}Y})} & & (Z,\bar{F}Z) \ar[u]_{(id_Z,id_{\bar{F}Z})} \\
		X \ar@/^1pc/[rrrr]^{g \circ f} \ar@(dl,ul)[]|{id_X} \ar[rr]_{f} &  & Y \ar@(dr,ur)[]|{id_Y} \ar[rr]_{g} & &  Z \ar@(dr,ur)[]|{id_Z}
	}
	\end{equation*}   
	\caption{$\int_{\mathbf{C}}\bar{\mathbf{F}}$ fibred on $\mathbf{C} \cong \mathbf{C}^{op}$; dotted arrows show $\mathbf{D}$ morphisms as actions}
	\label{fig:13fib}
\end{figure}

\begin{definition}[Concrete Right action induced by a functor]
	\label{defn8}
	Consider a strict contravariant functor $\bar{F}: \mathbf{C} \rightarrow \mathbf{D}$ between small categories with $(\mathbf{D}, U)$ being a concrete category over $\mathbf{Set}$ or a faithful $U:\mathbf{D} \rightarrow \mathbf{Set}$. Then ${\bar{\mathbb{F}}} = U \circ \bar{F}$ and $\int_{\mathbf{C}} \bar{\mathbb{F}}$ is a category with

	\begin{itemize}
		\item \textbf{objects}: the pairs $(X,x)$ where $X \in \mbox{Ob}(\mathbf{C})$ and $x \in \bar{\mathbb{F}}X = U(\bar{F}X)$
		
		\item \textbf{morphisms}: pairs $(f,x): (X,x) \rightarrow (Y,y)$ where $f:X \rightarrow Y \in \mathbf{C}$, $x = \bar{\mathbb{F}}f(y)$
		
		\item \textbf{identity}: for $(X,x)$, the morphism $\mathrm{id}_{(X,x)} = (\mathrm{id}_{X},x)$
		
		\item \textbf{composition}: $(g,{y}) \bullet (f,{x}) = (g \circ f,{x})$ since $x = \bar{\mathbb{F}}f\bar{\mathbb{F}}g(z)$ 
		
		\item \textbf{unit laws}: for $(f,{x})$, 
		$(\mathrm{id}_{Y},{y}) \bullet (f,{x}) = (f,{x}) = (f,{x}) \bullet (\mathrm{id}_{X},{x})$    
		
		\item \textbf{associativity}:  $(h,{z})\bullet ((g,{y})\bullet (f,{x}))=((h,{z})\bullet (g,{y}))\bullet (f,{x}) =(h,{z})\bullet (g,{y})\bullet (f,{x})$
	\end{itemize}
\end{definition}

\begin{figure}[!h]
	\begin{equation*}
	\xymatrix{
		(X,\bar{\mathbb{F}}f\bar{\mathbb{F}}g(z))  &  & (Y,\bar{\mathbb{F}}g(z)) \ar@{.>}[ll]  & & (Z,z) \ar@{.>}[ll] \\
		(X,x) \ar[u]^{(id_X,{x})} \ar@{>}[rrrru]^{(g \circ f,{x})} & & (Y,y)  & & (Z,z) \ar[u]_{(id_Z,{z})} \\
		(X,\bar{\mathbb{F}}f(y))  &  & (Y,\bar{\mathbb{F}}g(z)) \ar@{.>}[ll]_{U(\bar{F}f)}  & & (Z,z) \ar@{.>}[ll]_{U(\bar{F}g)} \\
		(X,x) \ar[u]^{(id_X,{x})}  \ar@{>}[rru]^{(f,{x})} & & (Y,y) \ar[u]^{(id_Y,{y})} \ar@{>}[rru]^{(g,{y})} & & (Z,z) \ar[u]_{(id_Z,{z})} \\
		X \ar@/^1pc/[rrrr]^{g \circ f} \ar@(dl,ul)[]|{id_X} \ar[rr]_{f} &  & Y \ar@(dr,ur)[]|{id_Y} \ar[rr]_{g} & &  Z \ar@(dr,ur)[]|{id_Z}
	}
	\end{equation*}   
	\caption{$\int_{\mathbf{C}}\bar{\mathbb{F}}$ fibred on $\mathbf{C} \cong \mathbf{C}^{op}$; dotted arrows denote concrete functions as right actions}
	\label{fig:rightconcretefib}
\end{figure}
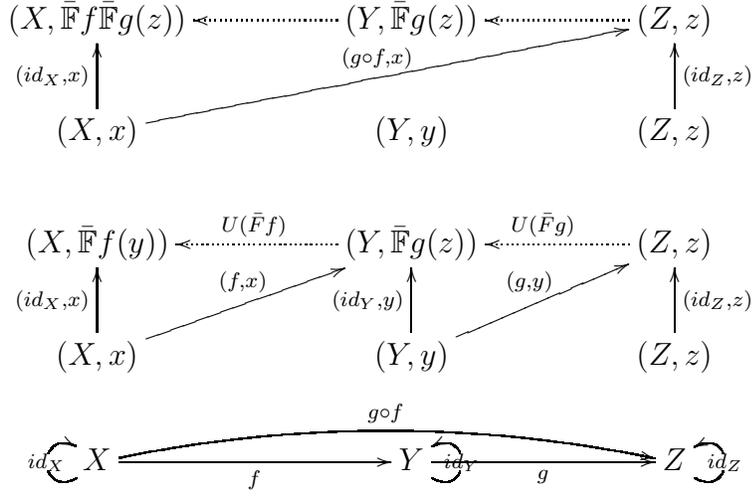 

The categories $\int_{\mathbf{C}^{op}} \bar{\mathbf{F}}$,$\int_{\mathbf{C}^{op}} \bar{\mathbb{F}}$ $\mathcal{X} \rtimes_{\mathbf{F}} \mathbf{C}$ (or $(\int_{\mathbf{C}^{op}} \bar{\mathbf{F}})^{op}$) $\mathcal{X} \rtimes_{\mathbb{F}} \mathbf{C}$ (or $(\int_{\mathbb{C}^{op}} \bar{\mathbb{F}})^{op}$) defined and visualized in this section generalize the concept of monoid action to category action. The generalization is in sense that there are simultaneous actions on multiple objects unlike a single object in a monoid action. More accurately in category action the object $\mathcal{X}$ (on which $\mathbf{F}$ or $\mathbb{F}$ defines an action) is a coproduct object inside the category $\mathbf{Cat}$ or $\mathbf{Set}$ respectively. This denotes actually the family of all the objects which lie in the image subcategory $F(\mathbf{C})$ and $\mathbb{F}(\mathbf{C})$ each treated as a category either trivially or usual categorification of its underlying set. More precisely, 
\begin{equation}
\mathcal{X} = \amalg_{X \in Ob(\mathbf{C})} \mathbf{F}(X) \quad \mathcal{X} = \amalg_{X \in Ob(\mathbf{C})} \mathbb{F}(X) 
\end{equation}
Symbolically this is same as $\mathbf{C} \times \mathcal{X} \rightarrow \mathcal{X}$. Of course the individual objects $FX$ in the category $F(\mathbf{C})$ are acted upon by all the arrows $f:X \rightarrow Y$ of $\mathbf{C}$ whose domain or source object is $X$ which can be captured by defining action set-theoretically or element wise. Indeed using standard functor definition we can observe that every $f$ defines a corresponding well-defined arrow $Ff: FX \rightarrow FY$ in the category $\mathbf{D}$. The action perspective in the context of symmetry will be revisited in the next paper~\cite{salilp2} of this sequel. However we briefly state and prove that the usual transformation groupoid $X \wquot G$ can be viewed as a base structured category. It is well-known in the groupoid literature \cite{brown87},\cite{weinstein} that a general group action gives rise to a transformation groupoid. More precisely,

\begin{definition}\cite{morton} Let $\phi : G \ra Aut(X)$, be a usual (set-theoretic) group action then transformation groupoid $X \wquot G$ is the groupoid consisting of:
	\begin{itemize}
		\item \textbf{Objects}: each element $x \in X $; denoted as $\mbox{Ob}(X \wquot G) = X$
		\item \textbf{Morphisms}: $(X \wquot G)(x,y) = (g,x) \in G \times X$, with $\phi_g(x)=y$
		\item \textbf{Composition}: $(g', \phi_g(x))\circ (g, x)= (g'g,x)$
	\end{itemize}
	\label{def:transfngroupoid}
\end{definition}

Now we show that every transformation groupoid is also a base structured category.  

\begin{proposition}
	\label{prop4}
	A classic transformation groupoid $X \wquot G$ is isomorphic to a base-structured category $\int_{\mathbf{C}^{op}} \bar{\mathbb{F}}$ where $\mathbf{C} = \mathbf{G}$ the group $G$ treated as one object category.   	
\end{proposition}
\begin{proof}
	The group $G$ can be viewed as a single object category $\mathbf{G}$ with an object $\star$ and $\mathbf{G}(\star,\star) = G$.
	But since $\mathbf{G}^{op} \cong \mathbf{G}$ we directly use the definition~\ref{defn8} where we consider a strict contravariant functor $\bar{F}: \mathbf{G} \rightarrow \mathbf{Set}$ between small categories with $(\mathbf{Set}, id)$ being modeled as a concrete category over itself or an underlying functor $U = id:\mathbf{Set} \rightarrow \mathbf{Set}$. Then ${\bar{\mathbb{F}}} = id \circ \bar{F} = \bar{F}$ and $\int_{\mathbf{G}} \bar{{F}}$ becomes a category with objects the pairs $(\star,x)$ where $\star \in \mbox{Ob}(\mathbf{G})$ and $x \in \bar{F}(\star)=X$, morphisms the pairs $(g,x): (\star,x) \rightarrow (\star,y)$ where $g:\star \rightarrow \star \in \mathbf{G}$, $x = \bar{F}g(y)$. The composition is given by $(g',{y}) \bullet (g,{x}) = (g' \circ g,{x})$ since $x = \bar{{F}}g\bar{{F}}g'(z)$. Finally comparing with the classic $X \wquot G$, we find that the objects of $X \wquot G$ are simply the objects of ${\int_{\mathbf{G}} \bar{{F}}}$ relabeled by dropping the first component while the morphisms are identical. Indeed we have $X \wquot G \cong \int_{\mathbf{G}} \bar{{F}}$ and in fact the transformation groupoid is just the special case of the general base structured category $\int_{\mathbf{C}^{op}} \bar{\mathbb{F}}$.
\end{proof}

We now state and prove a result interrelating the base structured categories.

\begin{proposition}
	\label{prop}
	Let $F:\mathbf{C} \rightarrow \mathbf{D}$ be any ordinary \textbf{abstract} 1-functor (thought of as a functor $\mathbf{F}: \mathbf{C} \xrightarrow{F} \mathbf{D} \xrightarrow{I} \mathbf{Cat}$ as defined earlier). Then the following categories are \textbf{abstractly} isomorphic,
	\begin{equation}
	\mathbf{C} \cong  (F,\mathbf{C},\mathbf{D}) \cong \mathcal{X} \rtimes_{\mathbf{F}} \mathbf{C}
	\end{equation}
	In addition if $\mathbf{C}^{op}$ is isomorphic to $\mathbf{C}$ then following categories are \textbf{abstractly} isomorphic
	\begin{equation}
	\mathbf{C} \cong  (F,\mathbf{C},\mathbf{D}) \cong \mathcal{X} \rtimes_{\mathbf{F}} \mathbf{C} \cong  \int_{\mathbf{C}} \bar{\mathbf{F}}
	\end{equation}
	Further additionally if $(\mathbf{D}, U)$ is a concrete category over $\mathbf{Set}$ then the following base-structured categories become \textbf{concretely} isomorphic,
	\begin{equation}
	(\mathbb{F},\mathbf{C},\mathbf{Set}) \cong  \int_{\mathbf{C}} \bar{\mathbb{F}} \cong \mathcal{X} \rtimes_{\mathbb{F}} \mathbf{C}
	\end{equation}
	Being subcategories of $\mathbf{C} \times \mathbf{D}$ or $\mathbf{C} \times \mathbf{Set}$ the base-structured categories have a first projection functor (which is restriction of the usual first projection functor) onto $\mathbf{C}$. The first projection functors sending $(F,\mathbf{C},\mathbf{D}), \mathcal{X} \rtimes_{\mathbf{F}} \mathbf{C},\int_{\mathbf{C}} \bar{\mathbf{F}}$ to $\mathbf{C}$ are the classic split (op)fibrations.
	The first projection functors sending $(\mathbb{F},\mathbf{C},\mathbf{Set}),\int_{\mathbf{C}} \bar{\mathbb{F}}, \mathcal{X} \rtimes_{\mathbb{F}} \mathbf{C}$ to $\mathbf{C}$ are the classic split opfibrations.
\end{proposition}
\begin{proof}
	First consider the category $(F,\mathbf{C},\mathbf{D})$.This is a well-defined category with objects the pairs $(X,FX),(Y,FY),...$ from $\mathbf{C}$ and $F\mathbf{C}$ and morphisms $\mathbf{F}((X,FX),(Y,FY))=\{(f,Ff):(X,FX)\rightarrow (Y,FY)\}$ considered pairwise from $\mathbf{C}$ and $F\mathbf{C}$. It is easy to observe that it is abstractly isomorphic to $\mathbf{C}$ by noting both the objects and the morphisms of each of these categories standing in a one-to-one correspondence to each other. In other words they are bijective on objects and on morphism sets. Same holds for $\mathcal{X} \rtimes_{\mathbf{F}} \mathbf{C}$ the way we have defined this category via Grothendieck construction or precisely as $(\int_{\mathbf{C}^{op}} \bar{\mathbf{F}})^{op}$ and consequently we have 
	\begin{equation}
	\mathbf{C} \cong  (F,\mathbf{C},\mathbf{D}) \cong \mathcal{X} \rtimes_{\mathbf{F}} \mathbf{C}
	\end{equation}

	Next if $\mathbf{C}^{op}$ is isomorphic to $\mathbf{C}$, we have a well-defined contravariant functor $\bar{F}: \mathbf{C} \rightarrow \mathbf{D}$ between small categories which could be thought of as $\bar{\mathbf{F}}: \mathbf{C} \rightarrow \mathbf{Cat}$ (with $\bar{\mathbf{F}} = I \circ \bar{F}$ and $I:\mathbf{D} \rightarrow \mathbf{Cat}$ as defined). Then $\int_{\mathbf{C}} \bar{\mathbf{F}}$ is a category as defined earlier in the sense of abstract right category action. Abstract isomorphism with $\mathbf{C}$ is easy to see and we have, 
	\begin{equation}
	\mathbf{C} \cong  (F,\mathbf{C},\mathbf{D}) \cong \mathcal{X} \rtimes_{\mathbf{F}} \mathbf{C} \cong \int_{\mathbf{C}} \bar{\mathbf{F}}
	\end{equation}
	
	It is crucial to note that the base-structured categories \textbf{cannot be made concretely isomorphic} to the base category $\mathbf{C}$ under any circumstances (such as if $\mathbf{C}$ and $\mathbf{D}$ are taken to be concrete and $F$ is made concrete) since there are additional components in the base-structured categories forgotten by the usual (first) projection functor. These are not bijective at the level of elements of underlying sets and consequently there is additional structure at level of sets which remains transparent to category theory supporting the intuition of trivial categorification (the objects which could be structured sets are treated simply as trivial categories concealing the structure from category theory). 
	
	Next let $(\mathbf{D}, U)$ be concrete category with some underlying category $\mathbf{X}$. Then we have a product category $(\mathbf{C} \times \mathbf{Set})$ (for simplicity we will assume these are constructs but the results hold for any underlying category $\mathbf{X}$ not necessarily $\mathbf{Set}$). The base-structured categories are each concrete subcategories of the $(\mathbf{C} \times \mathbf{Set})$. The concrete isomorphism is established by noting that objects of $(\mathbb{F},\mathbf{C},\mathbf{Set})$ are identical to $\int_{\mathbf{C}} \bar{\mathbb{F}}$ and $\mathcal{X} \rtimes_{\mathbb{F}} \mathbf{C}$ whereas the morphism $(f,\mathbb{F}f|_{x}): (X,x) \rightarrow (Y,y)$ is in bijection with $(f,y): (X,x) \rightarrow (Y,y)$ and $(f,x): (X,x) \rightarrow (Y,y)$ and concrete isomorphism is established as shown in Figure~\ref{fig:concreteiso}.
	\begin{figure}
		\[
		\begin{gathered}
		\xymatrix{
			(X,x)   &  & (Y,y)  \\
			(X,x) \ar[u]^{(id_X,x)} \ar@{>}[rr]^{(f,\mathbb{F}f|_{x})}  \ar@{.>}[rru]^{(f,y)} & & (Y,\mathbb{F}f(x)) \ar[u]^{(id_Y,y)}   \\
		}
		\end{gathered}
		\qquad
		\begin{gathered}
		\xymatrix{
			(X,\bar{\mathbb{F}}f(y))  \ar@{>}[rr]^{(f,\mathbb{F}f|_{x})} &  & (Y,y) \\
			(X,x) \ar[u]^{(id_X,x)}  \ar@{.>}[rru]^{(f,x)} & & (Y,y) \ar[u]^{(id_Y,y)}  \\
		} 
		\end{gathered}
		\] 
		\caption{Concrete isomorphisms between $(\mathbb{F},\mathbf{C},\mathbf{Set}), \mathcal{X} \rtimes_{\mathbb{F}} \mathbf{C}, \int_{\mathbf{C}} \bar{\mathbb{F}}$}
		\label{fig:concreteiso}
	\end{figure}
	
	Thus we have,
	\begin{equation}
	(\mathbb{F},\mathbf{C},\mathbf{Set}) \cong  \int_{\mathbf{C}} \bar{\mathbb{F}} \cong \mathcal{X} \rtimes_{\mathbb{F}} \mathbf{C}
	\end{equation}  
	
	Indeed all the three categories are isomorphic and have a first projection onto the category $\mathbf{C}$ defined as $p:(F,\mathbf{C},\mathbf{D}) \rightarrow \mathbf{C}$ where $p(X,FX) := X$ for all objects $(X,FX) \in Ob((F,\mathbf{C},\mathbf{D}))$ and $p(f,Ff) := f$ for all morphisms $(f,Ff) \in mor((F,\mathbf{C},\mathbf{D}))$. Using Proposition~\ref{prop2} it easily follows that the first projection functors sending $(F,\mathbf{C},\mathbf{D}), \mathcal{X} \rtimes_{\mathbf{F}} \mathbf{C},\int_{\mathbf{C}} \bar{\mathbf{F}}$ to $\mathbf{C}$ are the classic split (op)fibrations. On the other hand from Proposition~\ref{prop3} it easily follows that the first projection functors sending $(\mathbb{F},\mathbf{C},\mathbf{Set}),\int_{\mathbf{C}} \bar{\mathbb{F}}, \mathcal{X} \rtimes_{\mathbb{F}} \mathbf{C}$ are the classic split opfibrations. 
\end{proof}

The terminology base-structured categories reflects the fact that these categories have an abstract isomorphism with the base category $\mathbf{C}$. In base-structured categories the objects of $\mathbf{D}$ are trivially categorified which means in essence the structure only coming from the base objects and arrows. It is only when we consider the objects $FX$ (and therefore arrows $Ff$) as non trivial either as structured sets (which enables us to make use of set-theory along with category theory) or itself as category (which enables us to continue in category theory as classic fibrations with non-trivial vertical arrows). Since this is an obvious specialization of fibred categories it seemed appropriate to us to call the entire family consisting of these three categories as \textbf{base-structured categories} where the base is $\mathbf{C}$; since all arrows are Cartesian the essential structure of these categories is that of the base.

We will have more to say on these base-structured categories with their distinct perspective of a functor in applying category theory in fundamental applications such as symmetries and geometries in ~\cite{salilp2} and signal representation leading to arrow-theoretic redundancy in ~\cite{salilp3}.

\section{Conclusion}
\label{conclude}

The family of base structured categories; each characterizes a functor $F: \mathbf{C} \rightarrow \mathbf{D}$ in a unique way:

\begin{itemize}
	\item $(F,\mathbf{C},\mathbf{D})$ category characterizes a functor as a `\textbf{category structure preserving morphism}' or a graph of a functor.
	
	\item $\mathcal{X} \rtimes_{\mathbf{F}} \mathbf{C}$ or $(\int_{\mathbf{C}^{op}} \bar{\mathbf{F}})^{op}$ stems from the perspective of a functor as a multi-object \textbf{abstract left category action}.
	
	\item ${\int_{\mathbf{C}^{op}} \bar{\mathbf{F}}}$ stems from the perspective of a functor as a multi-object \textbf{abstract right category action}.
	
	\item $\mathcal{X} \rtimes_{\mathbb{F}} \mathbf{C}$ stems from the perspective of a functor as a multi-object \textbf{concrete left category action}.
	
	\item ${\int_{\mathbf{C}^{op}} \bar{\mathbb{F}}}$ stems from the perspective of a functor as a multi-object \textbf{concrete right category action}.
	
\end{itemize}

We proved that all of these categories are concretely isomorphic to each other yet only abstractly to the base and signify the fact that the arrows explicitly characterize the base structure since the vertical arrows of the total category are just the trivial identities. This gave rise to a distinct concept of \textbf{trivial categorification} and a new perspective for utilizing the potential of functor in
some fundamental applications which hitherto have been treating objects of $\mathbf{D}$ purely in a set theoretic way. This is dealt in the next papers, specifically~\cite{salilp2} expands the perspective of category action and connects precisely it to the generative theory of \cite{Leyton01} along with possible generalization of Klein geometries whereas~\cite{salilp3} works out a definition of redundancy in arrow-theoretic fashion and signal representation matched to its generative structure using these categories. This includes detailed discussions on compression and information analysis for a given signal to be represented and how many existing standards of image compression using differential encoding techniques could be reformulated precisely as special cases of such a category theoretic framework  .

\bibliography{p1.bib}
\bibliographystyle{acm}

\appendix
\section{Category and Functor}
\label{subsec:cf}
Having introduced by Eilenberg and Mac-Lane first in \cite{eilenberg-maclane}; the standard definition of a category and a functor has evolved somewhat to a form mostly found in today's standard reference text such as \cite{CWM}. For an excellent historical account of this field with an intuitive approach, the reader is referred to \cite{marquis}. Roughly speaking structures of a particular type with morphisms preserving this structure form a \textit{category}.

\begin{definition}\cite{CWM} An abstract category $\mathbf{C}$ consists of:
	\begin{itemize}
		\item \textbf{objects}: a collection $X,Y,Z ...$ denoted by $\mbox{Ob}(\mathbf{C})$
		
		\item \textbf{morphisms}:  for every pair $X,Y \in \mbox{Ob}(\mathbf{C})$, a collection $\mathbf{C}(X,Y)=\{f:X\to Y\mid X,Y\in\mbox{Ob}(\mathbf{C})\}$
		
		\item \textbf{identity}: for each $X \in \mbox{Ob}(\mathbf{C})$, a morphism $\mathbf{id}_X\ \mbox{or}\ \mathbf{1}_X\ :X\to X$
		
		\item \textbf{composition}: $\mathbf{C}(Y,Z)\ \times\ \mathbf{C}(X,Y)\mapsto\mathbf{C}(X,Z)$, i.e $(g,f)\mapsto g\circ f$
		
		\item \textbf{unit laws}: for a morphism $f:X \rightarrow Y$ we have $\mathbf{id}_Y\circ f=f=f\circ \mathbf{id}_X$
		
		\item \textbf{associativity}: for $X \xrightarrow[]{f} Y \xrightarrow[]{g} Z \xrightarrow[]{h} W$ we have $h\circ (g\circ f)=(h\circ g)\circ f$.
	\end{itemize}
	The schematic representation of an abstract category will be shown as 
	
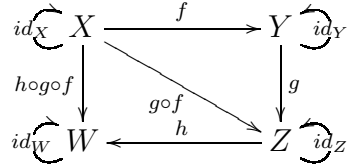
\begin{equation*}
	\xymatrix @R=0.4in @C=0.8in{
		X \ar@(dl,ul)[]|{id_X} \ar[r]^{f} \ar[d]_{h \circ g \circ f} \ar[dr]_{g \circ f} & Y \ar@(dr,ur)[]|{id_Y}  \ar[d]^{g} \\
		W \ar@(dl,ul)[]|{id_W} & Z \ar@(dr,ur)[]|{id_Z} \ar[l]_{h}
	}
	\end{equation*}
	
\end{definition}

Just like structures form categories of various sorts, categories themselves are also structures of a particular type with functors as morphisms. Indeed they form a category denoted as \textbf{Cat}. A functor makes coherence between different structures precise and this coherence of transformation is widely referred to as functoriality.

\begin{definition}\cite{CWM}
	Suppose $\mathbf{C}$ and $\mathbf{D}$ are categories.  A map $F:\mathbf{C}\rightarrow \mathbf{D}$ is a (covariant) functor consisting of:
	\begin{itemize}
		\item \textbf{object map}:  to every $X \in \mathbf{C}$, an object $F(X)\in\mbox{Ob}(\mathbf{D})$ 
		
		\item \textbf{morphism map}: to every morphism $f \in \mathbf{C}(X,Y)$, a morphism $F(f):F(X)\to F(Y)$ 
		
		\item \textbf{identity map}: for all $X \in \mathbf{C}$, $F(\mathbf{id}_X)=\mathbf{id}_{F(X)}$
		
		\item \textbf{composition map}: for all morphisms $X \xrightarrow[]{f} Y \xrightarrow[]{g} Z$ $F(f\circ g)=F(f)\circ F(g)$.

	\end{itemize}
	The schematic representation of a functor will be shown as 
	
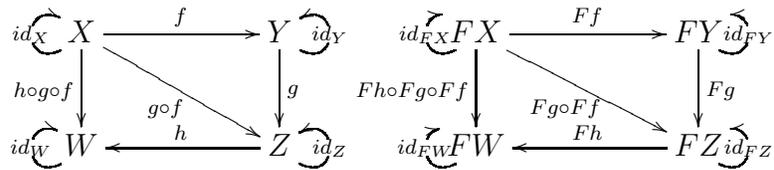
\begin{equation*}
	\xymatrix @R=0.4in @C=0.8in{
		X \ar@(dl,ul)[]|{id_X} \ar[r]^{f} \ar[d]_{h \circ g \circ f} \ar[dr]_{g \circ f} & Y \ar@(dr,ur)[]|{id_Y}  \ar[d]^{g} \\
		W \ar@(dl,ul)[]|{id_W} & Z \ar@(dr,ur)[]|{id_Z} \ar[l]_{h}
	}
	\xymatrix @R=0.4in @C=0.8in{
		FX \ar@(dl,ul)[]|{id_{FX}} \ar[r]^{Ff} \ar[d]_{Fh \circ Fg \circ Ff} \ar[dr]_{Fg \circ Ff} & FY \ar@(dr,ur)[]|{id_{FY}}  \ar[d]^{Fg} \\
		FW \ar@(dl,ul)[]|{id_{FW}} & FZ \ar@(dr,ur)[]|{id_{FZ}} \ar[l]_{Fh}
	}
	\end{equation*}
\end{definition}

\begin{remark}
	The concise schematic diagrams included along with the definitions are meant to graphically illustrate the various axioms stated. These are quite commonly known by the term `commutative diagrams'. Going forward in the sequel, we will make extensive use of these such diagrams. Otherwise and often difficult to grasp deep abstract concepts are simply made visually concrete and intuitive with their usage.
\end{remark}

\begin{definition}~\cite{joy}
	Let $\mathbf{X}$ be a category, then a concrete category over $\mathbf{X}$ is a pair $(\mathbf{C}, U)$, where $\mathbf{C}$ is a category and $U : \mathbf{C} \rightarrow \mathbf{X}$ is a faithful functor. Often $U$ will be called the underlying functor of the concrete category and $\mathbf{X}$ the underlying category for $(\mathbf{C}, U)$.
\end{definition}
In the standard category literature~\cite{joy} the underlying category is commonly called the base category which we avoid for obvious reasons. Concrete categories over $\mathbf{Set}$ are called \textbf{constructs} which are precisely the categories of structured sets and structure-preserving functions between them.

\begin{definition}~\cite{joy}
	Let $(\mathbf{C}, U)$ and $(\mathbf{D}, V)$ be some concrete categories over $\mathbf{X}$, then a concrete functor from
	$(\mathbf{C}, U)$ to $(\mathbf{D}, V)$ is a functor $F: \mathbf{C} \rightarrow \mathbf{D}$ with $U = V \circ F$ and denoted as $F : (\mathbf{C}, U) \rightarrow (\mathbf{D}, V)$.
\end{definition}
The condition $U = V \circ F$ implies that $U(X) = V(FX)$ (where $X$ is an object of $\mathbf{C}$) and $U(f) = V(Ff)$ (where $X$ is an object and $f$ is morphism of $\mathbf{C}$). Thus in the case of constructs, concrete functors ensure that the underlying sets of objects and set functions of morphisms in $\mathbf{C}$ are in bijection with $F\mathbf{C}$. This intuitively means that the information concerning the underlying sets and functions of objects and morphisms is also preserved by the concrete functors. This is distinctly in contrast to an usual \textbf{abstract} functor $F : \mathbf{C} \rightarrow \mathbf{D}$ between the concrete categories $(\mathbf{C}, U)$ and $(\mathbf{D}, V)$ which need not preserve the structure of underlying category $\mathbf{X}$.

\begin{definition}~\cite{joy}
	A functor $F : \mathbf{C} \rightarrow \mathbf{D}$ between categories is called an isomorphism if there is a functor $G : \mathbf{D} \rightarrow \mathbf{C}$ such that $G \circ F = id_{\mathbf{C}}$ and $F \circ G = id_{\mathbf{D}}$.
\end{definition}

\section{Product, Pullback and Limits}
\label{sec:pull}
Before we can assimilate the theory of fibration, we need to revisit the fundamental notion of limit in category theory. We begin with a brief recall of product which is one of the simplest limit to grasp. Then we discuss the actual concept of limit in an intuitive fashion using motivation from \cite{milewski}. Then we move on to define pullbacks as limits which are extremely fundamental to our entire work. Pullbacks are also directly related to the central concept of cartesian lifts utilized in the theory of fibration. The standard reference material for these limits is \cite{CWM},\cite{leinster2014},\cite{lawvere}.

\subsection{Product}
\label{sec:prod}
It is likely that the reader would be familiar with simple cartesian product of two arbitrary sets. This construction is a limit in disguise and provides an example of categorical product in the category $\mathbf{Set}$. The reader especially from signals and information theory background is strongly encouraged to refer \cite{lawvere} to get familiarized with basic constructions such as the cartesian product and overall set theory viewed from the perspective of category theory using the conceptual arrow approach. Here we simply recall the standard definition of product in the category theory to motivate the immediately following general concept of limit.
\begin{definition}[{\cite{CWM} and \cite{leinster2014}}]
	Let $\mathbf{C}$ be a category and let $A,B,C,D \in \mathbf{C}$ then a product of $A$ and $B$ consists of an object $C$ and projection maps $p_1$, $p_2$ with the property that for all objects such as $D$ with maps $q_1$, $q_2$ in $\mathbf{C}$, there exists a unique map $h: D \rightarrow C$ such that $p_1 \circ h = q_1$ and $p_2 \circ h = q_2$. The maps $p_1$ and $p_2$ are called the projections and we write $C$ as $A \times B$. The uniqueness of $h$ is often referred to as the universal property of a product.  
	
	\begin{example}
		Consider two categories $\mathbf{A}$, $\mathbf{B}$ which are objects of a category $\mathbf{Cat}$. This is a category with objects as small categories and arrows as functors. The category $\mathbf{A} \times \mathbf{B}$ is a category with objects the $<a,b>$ pairs  of objects $a$ of $\mathbf{A}$ and $b$ of $\mathbf{B}$. The arrow $<a,b> \rightarrow <a',b'>$ is the pair $<f,g>$ of arrows $f:a \rightarrow a'$, $g:b \rightarrow b'$ from $\mathbf{A}$ and $\mathbf{B}$ respectively. The reader can easily verify that this indeed is a category where everything is done component-wise Further it satisfies the universal property of a product as defined earlier.
	\end{example}
	\begin{figure}[h]	
		\begin{equation*}
		\xymatrix{
			D \ar@/_/[ddr]_{q_1} \ar@/^/[drr]^{q_2}
			\ar@{.>}[dr]|-{h}
			\\
			& A \times B \ar[d]^{p_1} \ar[r]_{p_2}
			& B 
			\\
			& A &
		}
		\end{equation*}
		\caption{Universal property of a product}
		\label{fig:product}
	\end{figure}
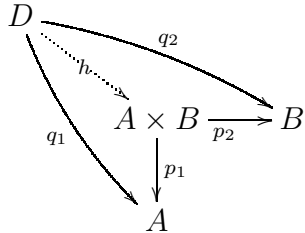
\end{definition}

\subsection{Limits}
\label{subsec:lim}
An overall intuitive, geometric or physical interpretation of a limit is that it embodies the structure or properties of a given diagram $D$ completely in a single object which contains exactly the same amount of information about the whole diagram neither more nor less. However the precise technical definition of this concept in classic reference \cite{CWM} utilizes the concept of natural transformation. Hence the reader might wish considering \cite{milewski} for more visual treatment of this concept. However if the reader is already familiar with concept of natural transformation then the following description recalls limit in a classical manner.

\begin{figure}[h]
	\begin{equation*}
	\xymatrix{
		& & & & & & \Delta_N \ar[rdd] \ar[lldd] \ar[ddd] & & \\
		& & & & & & & & \\
		X \ar@/^1pc/[uurrrrrr] \ar[rrr] \ar@/^1pc/[rrrr] \ar[rrd] & & & Z \ar@/^1pc/[uurrr] \ar@/^1pc/[rrrr] & DX \ar[rrr] \ar[rrd] & & & DZ\\
		& & Y \ar@/^1pc/[rrrr] \ar@/^2pc/[uuurrrr] \ar[ru] & & & & DY \ar[ru] & & \\
	}
	\end{equation*}
	\caption{A functor $D$ and a cone with vertex $N$ }
	\label{fig:cone}
\end{figure}
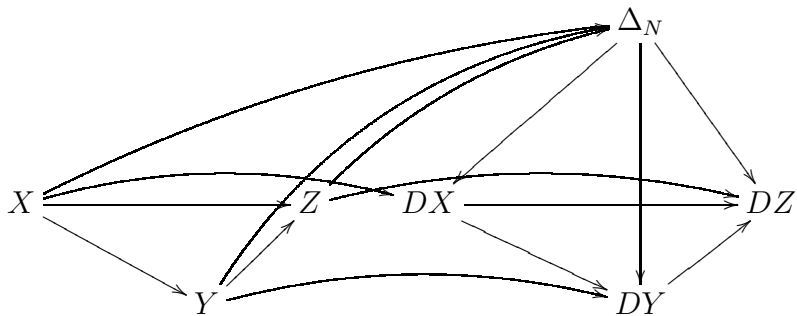
In general we could consider a category $\mathbf{J}$ which is small and often finite to be an abstraction of some finite pattern or structure or shape. Then a functor $D$ from $\mathbf{J}$ to some another category $\mathbf{C}$ is termed as $\mathbf{J}$-shaped diagram in $\mathbf{C}$. It is intuitively thought of indexing some collection of objects and morphisms in $\mathbf{C}$ patterned on $\mathbf{J}$. A constant functor $\mathbf{\Delta}_N$ from $\mathbf{J}$ to $\mathbf{C}$ sends every object to $N$ and every morphism to $id_N$. The natural transformation between these from $\mathbf{\Delta}_N$ to $D$ is called a \textbf{cone} with vertex $N$. This is shown in Figure~\ref{fig:cone} and visually resembles a cone, since the image of $\mathbf{\Delta}_N$ is the apex of a pyramid or cone whose sides are given by the components of the natural transformation with the image of $D$ forming the base of that cone.

\begin{definition}\cite{CWM} 
	A \textbf{limit} of a diagram $D$ is the universal cone (or limiting cone) of a diagram $D$.
	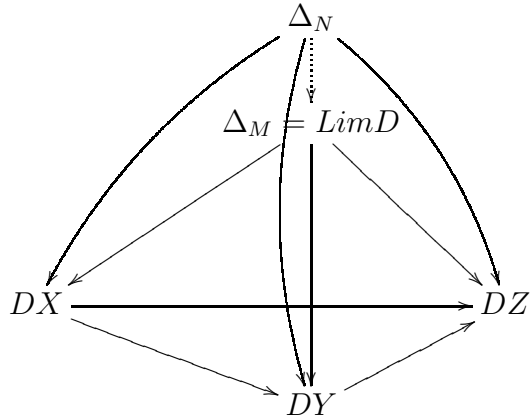
\begin{figure}[h]
		\begin{equation*}
		\xymatrix{
			&& \Delta_N \ar@{.>}[d] \ar@/^1pc/[rddd] \ar@/_1pc/[llddd] \ar@/_1pc/[dddd] && \\
			&& \Delta_M=Lim D \ar[rdd] \ar[lldd] \ar[ddd] && \\
			&&&\\
			DX \ar[rrr] \ar[rrd] &&& DZ\\
			&& DY \ar[ru]
		}
		\end{equation*}
		\caption{Limit of a diagram $D$ }
		\label{fig:limit}
	\end{figure}
\end{definition}

In-fact the cones based on a diagram $D$, form a category with objects as cones and morphisms as factorizing morphisms. Note that cones are entirely determined by their vertices and commuting triangles connecting two cones  factor through a morphism. This also means the universal cone is precisely the terminal object in this cone category. 

\subsection{Pullback}
\label{subsec:pull}
It is extremely crucial that the reader gets familiarized with the concept of pullback both heuristically as well as rigorously since the delicate concept of cartesian lift is completely based on this limit concept. Hence we shall also consider simple additional examples following the definition.

\begin{definition}\cite{CWM}
	Let $\mathbf{J}=(\bullet \rightarrow \bullet \leftarrow \bullet)$, then a functor $F:\mathbf{J} \rightarrow \mathbf{C}$ is a pair of arrows$
	\xymatrix{
		A \ar[r]^{f} & C & B \ar[l]_{g}
	}$
	with a common co-domain $C$. The cone over such a functor is a pair of arrows from a vertex $D$ such that the resulting square shown below (left) commutes. Then a universal cone or a limit (right) with vertex $P=A \times_{C} B$ satisfies the universal property with a unique $h:D \rightarrow A \times_{C} B$ as shown. The square formed by this universal cone is called a
	pullback square and its vertex is called a \textbf{pullback}, a `fibered product`, or a product over (the base object) $C$. Alternately it is said that f arises by pulling back along $g$,and $g'$ arises by pulling back $g$ along $f$.
	
	\begin{equation}
	\xymatrix{
		D \ar@/_/[ddr]_{q_1} \ar@/^/[drr]^{q_2}
		\ar@{.>}[dr]|-{h}
		\\
		& A \times_C B \ar[d]^{p_1=g'} \ar[r]_{p_2=f'}
		& B \ar[d]_g
		\\
		& A \ar[r]^f
		& C
	}
	\label{eq:pullback}
	\end{equation}
\end{definition}

The inner square formed in~\eqref{eq:pullback} is called a pullback or cartesian square. 
Here if we compare the commutative diagram with that of a product then there is an additional vertex which has a non-trivial object $C$ called base which gives it a significant power. Indeed if $C$ is a terminal object in the particular category in which commutative square~\eqref{eq:pullback} is considered then pullback precisely is the common product where we might drop the vertex altogether since the maps $f,g$ are unique without altering the essential definition.

Let us absorb the heuristic that a pullback object $P$ holds a relationship with object $B$ precisely in a way the object $A$ is related to the object $C$ by considering examples now.

\begin{example}
	In the category $\mathbf{FinSet}$ the objects are finite sets and arrows are the usual set functions. Then the pullback is $P = \{(a,b):a \in A, b \in B, f(a)=g(b)\}$. Thus $P$ is a particular subset of the cartesian product set $A \times B$. Intuitively as shown in the Figure~\ref{fig:pullback} this can be interpreted as $C$ indexing and partitioning both $A$ and $B$ and then $P$ is formed by the cartesian product of elements of $A$ and $B$ partition-wise or respecting the indexing structure. Hence note that the elements of $P$ are in the same proportion relative to elements in $B$, which is precisely the proportion of elements of $A$ relative to elements in $C$. Thus the heuristic that the pullback object holds a relationship with object $B$ precisely in a way the object $A$ is related to the object $C$ takes concrete form of relative count of elements in this example.
	
	\begin{figure}
		\begin{center}
			\begin{tikzpicture}[>=latex]

			\draw[thin] (2,2) circle (1);
			\draw[thin] (-2,2) circle (1);
			\draw[thin] (-2,-2) circle (1);
			\draw[thin] (2,-2) circle (1);
			
			\node at (3,3) {$A$};
			\node at (3,-3) {$C$};
			\node at (-3,3) {$P$};
			\node at (-3,-3) {$B$};
			
			\fill (-2,2) circle (0.07cm);
			\fill (-2.5,2.5) circle (0.07cm);
			\fill (-1.5,1.5) circle (0.07cm);

			\fill (-2,-2) circle (0.07cm);
			\fill (-2.5,-1.5) circle (0.07cm);
			\fill (-1.5,-2.5) circle (0.07cm);
			
			\node at (-2.7,-1.6) {$b_1$};
			\node at (-2.2,-2) {$b_2$};
			\node at (-1.7,-2.5) {$b_3$};
			
			\fill (2.5,1.5) circle (0.07cm);
			\fill (1.5,2.5) circle (0.07cm);
			\node at (1.8,2.5) {$a_1$};
			\node at (2.7,1.6) {$a_2$};
			
			\fill (2.5,-2.5) circle (0.07cm);
			\fill (1.5,-1.5) circle (0.07cm);
			\node at (1.8,-1.5) {$c_1$};
			\node at (2.4,-2.7) {$c_2$};
			
			\draw[->-] (-2,2) to (-2,-2); 
			\draw[->-] (-2.5,2.5) to (-2.5,-1.5); 
			\draw[->-] (-1.5,1.5) to (-1.5,-2.5); 
			
			\draw[->-] (-2,-2) to (1.5,-1.5);
			\draw[->-] (-2.5,-1.5) to (1.5,-1.5);
			\draw[->-] (-1.5,-2.5) to (2.5,-2.5);
			
			\draw[->-] (1.5,2.5) to (1.5,-1.5);
			\draw[->-] (2.5,1.5) to (2.5,-2.5);
			
			\draw[->-] (-2.5,2.5) to (1.5,2.5);
			\draw[->-] (-2,2) to (1.5,2.5);
			\draw[->-] (-1.5,1.5) to (2.5,1.5);
			
			\node at (0,2.8) {$p_1$};
			\node at (0,-2.8) {$g$};
			\node at (-2.8,0) {$p_2$};
			\node at (2.8,0) {$f$};
			
			\end{tikzpicture}
			
		\end{center}
		\caption{Heuristic of a Pullback in $\mathbf{FinSet}$}
		\label{fig:pullback}
	\end{figure}
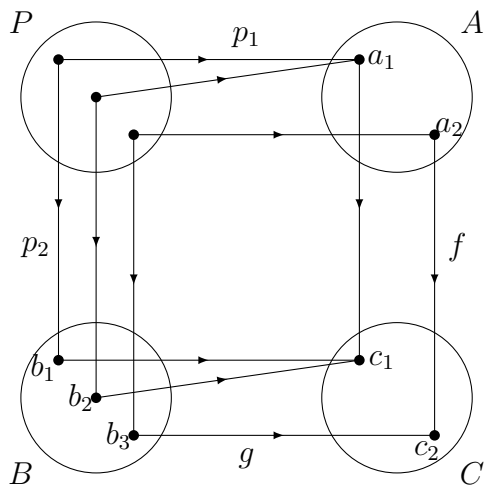
\end{example}
\begin{example}
	Inverse Image
	In the special case following last example, where the $g$ is an inclusion function, $p_1$ or $f'$ can be viewed as an appropriate restriction of $f$ as shown in~\ref{fig:inv}. Here $P = f^{-1}(B) = \{a:a \in A, f(a) \in B \}$.
	\begin{figure}[h]
		\begin{equation*}
		\xymatrix{
			f^{-1}(B) \ar[r]^{f'} \ar@{^{(}->}[d] & B \ar@{^{(}->}[d] \\
			A \ar[r]_{f} & C
		}
		\end{equation*}
		\caption{Inverse image of a function as a pullback}
		\label{fig:inv}
	\end{figure}
\end{example}

\section{Fibred category theory }
\label{fib1}
Having recalled the basic category,functor and pullback we are ready to study the fibred category theory for our work. Although the original reference is [\cite{SGA1} Section VI]; the easier references for us have been \cite{BJ}, \cite{vistoli}. We also have gained much from \cite{barnet}. This section deals explicitly with the traditional fibred category theory. The formal theory arise from the work of Grothendieck in algebraic geometry and deals with indexing of category using another category. In the theory there are two equivalent formulations viz. \textit{category-indexed category} and \textit{categorical fibration}. They are a natural generalization of the concept of set-indexed set to categories, hence it will be easier for us to first consider briefly the essence of set-indexed set for our work.   

\subsection{Set-Indexed Sets}
There are in general two different equivalent ways of presenting the concept of set-indexed sets as studied in \cite{BJ}.

\begin{enumerate}
	\item \textbf{Pointwise or Split Indexing:} 
	In this form of indexing, to every element $i \in I$ some ordinary set $X_i$ is assigned. In the language of category theory, this can be formulated using a functor $\Psi: I \rightarrow \textbf{Set}$. Notice that since in the category $\mathbf{Set}$ (of all small sets with functions) every singleton is an object, we can form a functor from a (discrete) subcategory $I$ with objects corresponding to the elements of set $I$ to the category $\mathbf{Set}$. Note that by the definition of functor the image (category) will precisely contain objects $X_i$.
	
	\item \textbf{Display Indexing:}
	This is simply given by an ordinary function $f: X \rightarrow I$. Heuristically this decomposes or partitions the whole set $X$ into particular subsets, which is specified by $f$. The subsets are termed as fibers over the elements of $I$.
\end{enumerate}

It is important to realize here that the specification of $f$ in turn specifies how $X$ is viewed structurally relative to $I$. Since this particular case deals only with objects, with all arrows trivial, it only specifies how $X$ is partitioned or more precisely $X = \amalg_i X_i$ where the union is necessarily disjoint. However with generalization to arbitrary arrows or general categories it has a far reaching powerful interpretation on how structures are viewed relatively.To forget the indexing or base in this case, simply amounts to forgetting the particular partition induced by this index set.  

Further display indexing is simply characterized by a general set function in which fibers are always disjoint (being inverse images of this function), it readily generalizes to categories as compared to point-wise indexing. The function displays the family over a base, hence the terminology.
\begin{example}
	The constant family (or trivial fibration) consists of all fibres being isomorphic to one set $X_i$.
\end{example}
\begin{example} 
	The arrow category $\mathbf{Set}^{\rightarrow}$ which contains set-indexed sets or a family as objects with morphisms preserving the indexing structure. The category $\mathbf{Set}^{\rightarrow}$ consists of:
	\begin{itemize}
		\item \textbf{objects}: a collection of all set-indexed sets of form $\left(
		\def\objectstyle{\scriptstyle}
		\def\labelstyle{\scriptstyle}
		\vcenter{\xymatrix @-1.2pc
			{
				X \ar[d]^{f} \\
				I
		}}
		\right)$
		
		\item \textbf{morphisms}: for every pair $\left(
		\def\objectstyle{\scriptstyle}
		\def\labelstyle{\scriptstyle}
		\vcenter{\xymatrix @-1.2pc
			{
				X \ar[d]^{f} \\
				I
		}}
		\right)$
		$\left(
		\def\objectstyle{\scriptstyle}
		\def\labelstyle{\scriptstyle}
		\vcenter{\xymatrix @-1.2pc
			{
				Y \ar[d]^{g} \\
				J
		}}
		\right)$
		a morphism $(u,h)$ where $u:I \rightarrow J$ and $h:X \rightarrow Y$ with $g \circ h=u \circ f$
		
		\item \textbf{identity}: for each $\left(
		\def\objectstyle{\scriptstyle}
		\def\labelstyle{\scriptstyle}
		\vcenter{\xymatrix @-1.2pc
			{
				X \ar[d]^{f} \\
				I
		}}
		\right)$, a morphism $(\mathbf{id}_I,\mathbf{id}_X)$ or pair of identities from \textbf{Set}
	\end{itemize}
	satisfying the usual axioms of a category such as composition, unit laws and associativity. The special case of $\mathbf{Set}/{I}$ widely known as \textit{Slice Category} is obtained for a fixed index set $I$.
\end{example}

\subsection{Change of Base}
For a given family,
$\left(
\def\objectstyle{\scriptstyle}
\def\labelstyle{\scriptstyle}
\vcenter{\xymatrix @-1.2pc
	{
		X \ar[d]^{f} \\
		I
}}
\right)$
the change of base is achieved using the pull-back operation as shown in the Figure~\ref{fig:basechange}
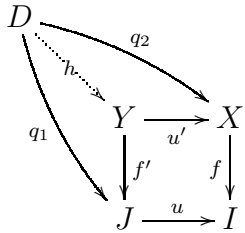
\begin{figure}[h]
	\begin{equation*}
	\xymatrix{
		D \ar@/_/[ddr]_{q_1} \ar@/^/[drr]^{q_2}
		\ar@{.>}[dr]|-{h}
		\\
		& Y \ar[d]^{f'} \ar[r]_{u'}
		& X \ar[d]_f
		\\
		& J \ar[r]^u
		& I
	}
	\end{equation*}
	\caption{Change of base for set-indexed sets.}
	\label{fig:basechange}
\end{figure}

Denote this pull-back of $f$ against $u$ as $Y$ which is $J \times_I X$. Given any function $f:X \rightarrow I$, we can interpret this as a decomposition of the set $X$ into subsets $f^{-1}(i)$, indexed by the elements $i \in I$. This way of looking at a function is the essential essence underlying the idea of fibration where $X$ is termed as being fibered on $I$ and $f^{-1}(i)$ are the fibres of $X$ at $i \in I$. Remember that $I$ which will be referred to as base in the general categorical framework, sets up a partition of $X$ as the fibres are necessarily disjoint.

\subsection{Fibration}
Generalizing the concept of collection of sets varying over an index set to categories, we get the concept of category indexed categories corresponding to point-wise indexing and fibration corresponding to display indexing. The theory of fibred categories makes this concept of collections of categories that vary over a base category mathematically precise. It is also seen as as ordinary category theory over some base category (thought of as a universe). First we consider fibration and later we study category-indexed categories. The fundamental construction called \textbf{Grothendieck completion} turns category-indexed-categories into fibration is explored thereafter.

We begin with recalling the definition of fibration.
\begin{definition}[\cite{BJ}]
	Let $P:\mathbf{E} \rightarrow \mathbf{B}$ be a usual functor;
	\begin{itemize}
		
		\item A morphism $f:X \rightarrow Y$ in $\mathbf{E}$ is cartesian over $u:I \rightarrow J$ in $\mathbf{B}$ if $Pf=u$ and every $g:Z \rightarrow Y$ in $\mathbf{E}$ for which one has $Pg=uw$ for some $w:PZ \rightarrow I$, uniquely determines an $h:Z \rightarrow X$ in $\mathbf{E}$ above $w$ with $fh=g$.
		
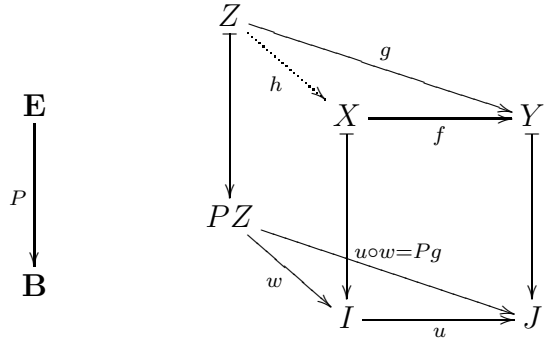
		\begin{equation*}
		\xymatrix{ & \\
			\mathbf{E} \ar[dd]_{P} \\ 
			& \\
			\mathbf{B}}
		\qquad
		\xymatrix{ Z \ar@{|->}[dd] \ar@{.>}[rd]_{h} \ar[rdrr]^{g} & & \\
			& X \ar@{|->}[dd] \ar[rr]_{f} & & Y \ar@{|->}[dd] \\
			PZ \ar[rd]_{w} \ar[rdrr]^{u \circ w =Pg}  & & \\
			& I \ar[rr]_{u} & & J }
		\end{equation*}
		
		\item The functor $P:\mathbf{E} \rightarrow \mathbf{B}$ is a fibration (or a fibred category) if for every $Y \in \mathbf{E}$ and $u:I \rightarrow PY$ in $\mathbf{B}$, there is a cartesian morphism $f:X \rightarrow Y$ in $\mathbf{E}$ above $u$.
	\end{itemize}
\end{definition}

This uniqueness resulting in $\mathbf{E}$ is termed as the unique lifting property. The concept of cartesian lift is very closely related to pullbacks in ordinary categories. To expose this connection heuristically we use the analogy of mixed categories from \cite{barnet}. For this consider a category in which we mix objects and arrows from $\mathbf{E}$ and $\mathbf{B}$ as specified:
\begin{itemize}
	\item \textbf{objects}  All objects from both $\mathbf{E}$ and $\mathbf{B}$ i.e $X,Y,.. \in \mathbf{E}$ and $I,J,.. \in \mathbf{B}$
	\item \textbf{morphisms} All $f,g,.. \in \mathbf{E}$, $u,v,.. \in \mathbf{B}$ with $P_X: X \rightarrow I$ where $PX = I$ quotiented by $Pf = u$
	\item \textbf{identity} identities from both $\mathbf{E}$ and $\mathbf{B}$
	\item \textbf{composition} usual compositions and cross compositions such as $u \circ P_X$ and $P_Y \circ f$
\end{itemize}

with usual unit laws and associativity law. Note that by fiat we declare that the (mixed) diagrams of the type shown below truly commute only when $Pf=u$ or $u \circ P_X =  P_Y \circ f$ as shown in Figure~\ref{fig:comm}
\begin{figure}[h]
	\begin{equation*}
	\xymatrix{
		X \ar[r]^{f} \ar@{|->}[d]_{P_X} & Y \ar@{|->}[d]^{P_Y} \\
		I \ar[r]_{u} & J
	}
	\end{equation*}
	\caption{Commutative squares in the mixed category}
	\label{fig:comm}
\end{figure}
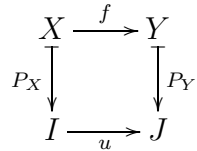
where ${P_X}$ is the arrow $X \mapsto I$, the usual restriction of the functor $P$ to $X$ whereas $f$ and $u$ are the general arrows in $\mathbf{E}$ and $\mathbf{B}$ respectively.

Now note that for $f:X \rightarrow Y$ to be cartesian precisely means that the resulting square is a pullback or cartesian square in this mixed union category. The usual uniqueness of this pullback object $X$ could be expressed using the Figure ~\ref{fig:liftaspull}
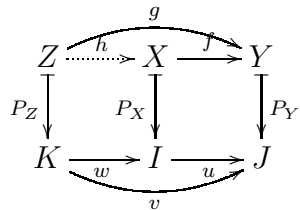
\begin{figure}[h]
	\begin{equation*}
	\xymatrix{
		Z \ar@/^1pc/[rr]^{g} \ar@{|->}[d]_{P_Z} \ar@{.>}[r]^{h} & X \ar[r]^{f} \ar@{|->}[d]_{P_X} & Y \ar@{|->}[d]^{P_Y} \\
		K \ar@/_1pc/[rr]_{v} \ar[r]_{w} & I \ar[r]_{u} & J
	}
	\end{equation*}
	\caption{Cartesian lift as a pullback in mixed category}
	\label{fig:liftaspull}
\end{figure}
which precisely matches the definition of the cartesian lift of $u$ at $Y$.

Hence most of the properties of cartesian arrows which we gather now for our work in the sequel, follow naturally from the usual properties of the pullback in ordinary categories. In particular the intuition for the cartesian lift of a general arrow precisely follows from intuition regarding the pullback we discussed earlier. This essentially says that $X$ relative to $Y$ is exactly similar to $I$ relative to $J$. In fibration $\mathbf{E}$ is thought of as being above $\mathbf{B}$.

\subsection{Basic Properties of Fibration}
Here we recall certain standard results pertaining to the total category $\mathbf{E}$ which will be required later in the sequel.
\begin{proposition}\cite{BJ}
	\label{prop:vert_cart}
	Every morphism in $\mathbf{E}$ factors as a vertical map followed by a cartesian one.
\end{proposition}
\begin{proof}
	Consider an arbitrary morphism $g:Z \rightarrow Y$ in the total category $\mathbf{E}$. Using definition of a functor $P$, let $PZ=I$ and $PY=J$ and $P(g)=u$ where $u:I \rightarrow J$. But since $P:\mathbf{E} \rightarrow \mathbf{B}$ is also a fibration it follows that $Z \in \mathbf{E}_I $ and $Y \in \mathbf{E}_J $ and there must exist cartesian lift $u$ at $Y$. Let this cartesian arrow above $u$ be $f$ then we have a unique $h$ with $g=f \circ h$
	\begin{displaymath}
	\xymatrix { Z \ar [rd]^{g} \ar @{.>}[d]^{h} &  \\
		X \ar [r]^{f} & Y }
	\end{displaymath}
	But since both $f$ and $g$ lie above $u$, $Z$ and $X$ must lie above $I$ or in the fiber $\mathbf{E}_I$. Thus the resulting unique $h$ is a vertical map.
\end{proof}
Proposition~\ref{prop:vert_cart} can be interpreted as saying that the total structure in the fibred category is a fusion of horizontal and vertical structure along with the action of horizontal on vertical. 

\begin{proposition}\cite{BJ}
	Composition of cartesian arrows in $\mathbf{E}$ is also a cartesian arrow.
\end{proposition}
\begin{proof}
	Let $f$ and $g$ be the cartesian arrows above some $u$ and $v$ in the base category. Then for some arbitrary arrow $j$ from a fixed object $W$ to $Y$ we have a unique $h_1$ such that $f \circ h_1 = j$. 
	\begin{equation*}
	\xymatrix{ & W \ar@{.>}[ld]_{h_2} \ar@{.>}[dd]^{h_1} \ar[rddr]^{j} & \\  
		Z \ar[rd]_{g} \ar[rdrr]^{f \circ g} & & \\
		& X \ar[rr]^{f} & & Y \\
	}
	\end{equation*}
	But since $g$ is cartesian, for the same fixed object $W$ and arrow $h_1$ to $X$, we have a unique $h_2$ such that $g \circ h_2 = h_1$. Noting that all inner triangles commute, we have for an arbitrary arrow $j$ a unique $h_2$ such that $(f \circ g) \circ h_2 = j$ proving that $f \circ g$ is cartesian. 
\end{proof}

\begin{proposition}\cite{BJ}
	In $\mathbf{E}$ Cartesian map above an isomorphism is also an isomorphism.
\end{proposition}
\begin{proof}
	Let a $\mathbf{B}$-arrow $u:I \rightarrow J$ be an isomorphism. This implies by the definition of isomorphism that there exists another $\mathbf{B}$-arrow $v:J \rightarrow I$ such that $u \circ v = id_J$ and $v \circ u = id_I $. Now let the cartesian map above u be $f:X \rightarrow Y$ for a given $Y$ and the cartesian map above v for the $X$ be $g:Z \rightarrow X$. Then the vertical map $f \circ g$ is also cartesian as the composition of cartesian maps is cartesian and since $u \circ v = id_J$ therefore it must be an identity arrow or $f \circ g = 1_Y$  implying $Z=Y$.
	\begin{displaymath}
	\xymatrix { & Z \ar [ld]_{g} \ar@{>}[d]^{f \circ g} &  \\
		X \ar [r]^{f} & Y }
	\end{displaymath}
	similarly in the other direction one can show that $g \circ f = 1_X$. Thus the cartesian lift of an isomorphism is also an isomorphism.
\end{proof}

\subsection{Cloven and Split Fibration}
\label{cs}
Since mostly we will need strict functors rather than pseudo-functors, the reader may skip this section without any harm. We include these two kinds of fibrations to get familiarized with the classic fibred category theory associated with pseudo-functors. The definition of fibration guarantees that the cartesian lift (or the pullback completion) for every possible $u$ and $Y$ exists but keeps the choice unspecified. Indeed such a pullback object is only unique up to a unique isomorphism, so one has to explicitly make a choice of cartesian lift. When the choice is made (upto vertical isomorphism) as explained next, the fibration is called \textbf{cloven} and it equivalently corresponds to a pseudo-functor or category-valued presheaf $\textbf{B}^{op} \rightarrow \textbf{Cat}$.

Fist let us fix the notation and diagram mostly following \cite{BJ} for the cartesian lift of $u$ at $X$ as 

\begin{equation*}
\xymatrix{
	u^*(X) \ar@{.>}[r]^{\bar{u}(X)} & X \\
}
\end{equation*}

Then for a general vertical morphism $f:Y_1 \rightarrow Y_2$ between two arbitrary objects in the fibre over $J$, we have
\begin{equation*}
\xymatrix{
	u^*(Y_1) \ar@{.>}[r]^{\bar{u}(Y_1)} \ar@{-->}[d]_{u^*(f)} & Y_1 \ar@{->}[d]^{f} \\
	u^*(Y_2) \ar@{.>}[r]_{\bar{u}(Y_2)} & Y_2 \\
	I \ar[r]^{u} & J
}
\end{equation*}

Thus every map $u:I \rightarrow J$ in $\textbf{B}$ determines a functor $u^*$ in a reverse direction from the whole fibre $\textbf{E}_J$ to the whole fibre $\textbf{E}_I$. Such a functor is referred to as \textbf{change-of-base} or \textbf{pullback functor}.

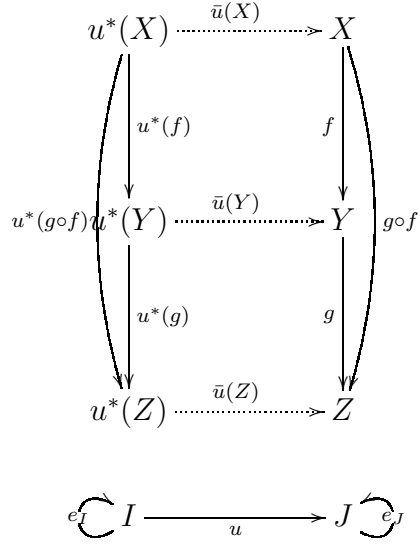
\begin{figure}[!h]
	\begin{equation*}
	\xymatrix{
		u^*(X) \ar@/_1pc/[dddd]_{u^*(g \circ f)} \ar[dd]^{u^*(f)} \ar@{.>}[rr]^{\bar{u}(X)} &  & X \ar@/^1pc/[dddd]^{g \circ f} \ar[dd]_{f} \\
		& & \\
		u^*(Y) \ar[dd]^{u^*(g)} \ar@{.>}[rr]^{\bar{u}(Y)} &  & Y \ar[dd]_{g}  \\
		& & \\
		u^*(Z) \ar@{.>}[rr]^{\bar{u}(Z)} &  & Z \\
		I \ar@(dl,ul)[]|{e_I} \ar[rr]_{u} &  & J \ar@(dr,ur)[]|{e_J}
	}
	\end{equation*}
	\caption{Pullback functor from $\textbf{E}_J$ to $\textbf{E}_I$ determined by cartesian lifts}
	\label{fig:indexedcat}
\end{figure}

For identities and composition of arrows in general we have canonical natural transformations as shown in Figures~\ref{fig:cloven1},~\ref{fig:cloven2} respectively.
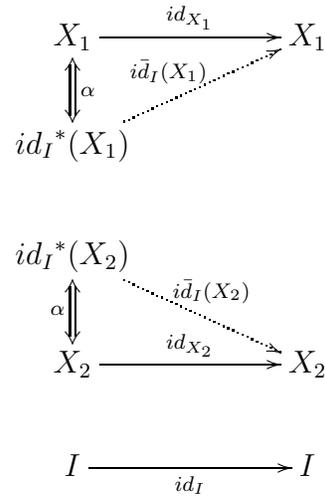
\begin{figure}[!h]
	\begin{equation*}
	\xymatrix{
		X_1 \ar@2{<->}[d]^{\alpha} \ar@{>}[rr]^{id_{X_1}} &  & X_1 \\
		{id_I}^*(X_1) \ar@{.>}[rru]^{\bar{id_I}(X_1)} & & \\
		{id_I}^*(X_2) \ar@{.>}[rrd]^{\bar{id_I}(X_2)} & & \\
		X_2 \ar@2{<->}[u]^{\alpha} \ar@{>}[rr]^{id_{X_2}} &  & X_2 \\
		I \ar[rr]_{id_I} &  & I
	}
	\end{equation*}
	\caption{Cartesian lift of identity in a cloven fibration}
	\label{fig:cloven1}
\end{figure}

\begin{figure}[!h]
	\begin{equation*}
	\xymatrix{
		u^*v^*(Z) \ar@{>}[rr]^{\bar{u}[v^*(Z)]} &  & v^*(Z) \ar@{>}[rr]^{\bar{v}(Z)} & & Z \\
		(v \circ u)^*(Z) \ar@2{<->}[u]^{\beta} \ar@{.>}[rrrru]_{\bar{(v \circ u)}(Z)} & &  & & & \\
		I \ar@/^1pc/[rrrr]^{v \circ u} \ar[rr]_{u} &  & J \ar[rr]_{v} & & K
	}
	\end{equation*}
	\caption{Cartesian lift of composite arrow in a cloven fibration}
	\label{fig:cloven2}
\end{figure}

When the cloven fibration is such that induced pullback functors make the natural transformations actual identities, then its is called a \textbf{split fibration}. Since split fibrations guarantee these functoriality conditions, the pseudo-functor of the cloven case actually reduces to becoming a true functor. Thus split fibrations are said to be well-behaved and comfortable to work with when such a case is possible. Fortunately most of the cases we will encounter for the purposes of signal representation would be split and would correspond to strict functors such as $L^2$. This is because from an applied perspective we mostly deal with linear spaces having single object in the fibres when they are viewed as categories. Yet it should be noted that often such equalities are not guaranteed in general for categories and the cleavages result in natural isomorphisms (instead of equalities) as discussed.
For more complete and delicate issues the reader is referred to \cite{BJ}.

\subsection{B-Indexed Category and Grothendieck Construction}
\label{cic}
We now recall the definition of equivalent notion of point-wise indexing of sets in the categories which goes by the terminology of category-indexed categories. 
\begin{definition}\cite{BJ}
	A $\mathbf{B}$-indexed category is a pseudo functor $\Psi:\mathbf{B}^{op} \rightarrow \mathbf{Cat}$. It maps each object $I \in \mathbf{B}$ to a category $\Psi(I)$ and each morphism $u:I \rightarrow J \in \mathbf{B}$ to a functor $\Psi(u):\Psi(J) \rightarrow \Psi(I)$ with direction reversed. Such a functor $\Psi(u)$ is denoted by $u^*$, and a pseudo-functor unlike a strict functor involves natural isomorphisms $\alpha_I: id \approxeq (id_I)^*$ and $\beta_{u,v}:u^*v^* \approxeq (v \circ u)^*$ for usual objects $I,J,K$ and arrows $u,v$ in $\mathbf{B}$ satisfying the classic coherence conditions.
\end{definition}

Because of the natural isomorphisms for a general cloven fibration, indexed category must satisfy certain coherence conditions common in categories as shown in Figure~\ref{fig:coherence1} and Figure~\ref{fig:coherence2}. 
\begin{figure}[!h]
	\begin{equation*}
	\xymatrix{
		& u^* \ar@{=}[d] \ar[dl]_{\alpha_Iu^*} \ar[dr]^{u^*\alpha_J} & \\
		(id_I)^*u^* \ar[r]_{\beta_{id_{I},u}} & u^* & u^*(id_J)^* \ar[l]^{\beta_{u,id_{J}}}
	}
	\end{equation*}
	\caption{Coherence conditions for $u:I \rightarrow J$}
	\label{fig:coherence1}
\end{figure}

\begin{figure}[!h]
	\begin{equation*}
	\xymatrix{
		u^*,v^*,w^* \ar[r]^{u^*\beta_{v,w}} \ar[d]_{\beta_{u,v}w^*} & u^*(w \circ v)^* \ar[d]^{\beta_{u,w \circ v}} \\
		(v \circ u)^*w^* \ar[r]_{\beta_{v \circ u,w}} & (w \circ v \circ u)^*
	}
	\end{equation*}
	\caption{Coherence conditions for $I \rightarrow^{u} J \rightarrow^{v} K \rightarrow^{w} L$}
	\label{fig:coherence2}
\end{figure}

Next we recall the classic definition of Grothendieck construction. Immediately follows Proposition~\ref{strict} which will facilitate the understanding of why and how this construction turns a pseudo-functor into a fibration. But we demonstrate this only for the case of a strict functor which suffices for our requirements in an applied context. 

\begin{definition}\cite{BJ}
	Let $\Psi: \mathbf{B}^{op} \rightarrow \mathbf{Cat}$ be an indexed category. Then \textbf{Grothendieck Construction} $\int_{\mathbf{B}}\Psi$ of $\Psi$ is the total category consisting of:
	\begin{itemize}
		\item \textbf{objects} $(I,X)$ where $I \in \mathbf{B}$ and $X \in \Psi(I).$
		\item \textbf{morphisms} $(I,X) \rightarrow (J,Y)$ are the pairs $(u,f)$ with $u:I \rightarrow J$ in $\mathbf{B}$ and $f:X \rightarrow u^*(Y)=\Psi(u)(Y)$ in $\Psi(I)$.
		\item \textbf{identity} $(I,X) \rightarrow (I,X)$ is pair $(id,\alpha_{I}(X))$, with $\alpha_{I}$ the natural isomorphism $id_{\Psi(I)} \approxeq (id_I)^*$.
		\item \textbf{composition}
		$\xymatrix{
			(I,X) \ar[r]^{(u,f)} & (J,Y) \ar[r]^{(v,g)} & (K,Z) 
		}$
		where
		$X \rightarrow^{f} u^*(Y) \rightarrow^{u^*(g)} u^*v^*(Z) \approxeq (v \circ u)^*(Z)$
	\end{itemize}
	satisfying the usual unit laws and associativity axioms with the coherence conditions guaranteeing equalities for identity and composition as shown earlier in Figures~\ref{fig:coherence1} and \ref{fig:coherence2}.  
\end{definition}
The schematic is for general construction with cleavage will be denoted as 
\begin{equation*}
\xymatrix{
	(I,X_1) \ar[d]_{(id,f_1)} \ar@{>}[rr]^{(u,f_1)} &  & (J,Y_1) \\
	(I,u^*(Y_1)) \ar@{.>}[rru]^{\bar{u}(J,Y_1)} & & \\
	(I,u^*(Y_2)) \ar@{.>}[rrd]^{\bar{u}(J,Y_2)} & & \\
	(I,X_2) \ar[u]^{(id,f_2)} \ar@{>}[rr]^{(u,f_2)} &  & (J,Y_2) \\
	I \ar@(dl,ul)[]|{id} \ar[rr]_{u} &  & J
}
\end{equation*}

\begin{equation*}
\xymatrix{
	(I,X) \ar[d]^{f} \ar@{>}[rr]^{(u,f)} &  & (J,Y) \ar[d]^{g} \ar@{>}[rr]^{(v,g)} & & (K,Z) \\
	u^*(Y) \ar[d]^{u^*(g)} \ar@{.>}[rru]^{\bar{u}(Y)} & & v^*(Z) \ar@{.>}[rru]^{\bar{v}(Z)} & & & \\
	u^*v^*(Z) \ar@2{<->}[d]^{\beta} \ar@{.>}[rru]^{\bar{u}[v^*(Z)]} & & \\
	(v \circ u)^*(Z) \ar@{.>}[rrurruu]_{\bar{(v \circ u)}(Z)} & & \\
	I \ar@/^1pc/[rrrr]^{v \circ u} \ar@(dl,ul)[]|{e_I} \ar[rr]_{u} &  & J \ar@(dr,ur)[]|{e_J} \ar[rr]_{v} & & K \ar@(dr,ur)[]|{e_K}
}
\end{equation*}
Note that in the general cloven schematic of fibred category in the composition we have shown only second component of all objects and arrows except for the first row for simplicity to avoid clutter.
The reader must exercise caution with notations since in the fibration definition we have made use of the notation $X$ for an object above $I$ in $\mathbf{E}_I$; however the same notation $X$ for an object in $\Psi(I)$ will result into an object $(I,X)$ above $I$ when translated into the fibration terminology. The reader must bear this distinction in mind even if we slightly abuse and use the similar notations such as $X$ in both cases of indexed categories and fibration for simplicity.

\begin{proposition}\cite{BJ},\cite{vistoli},\cite{barnet}
	\label{strict}
	A fibred category over $\mathbf{B}$ with a cleavage defines a pseudo-functor $\mathbf{B}^{op} \rightarrow \mathbf{Cat}$. Conversely from every pseudo-functor on $\mathbf{B}$ we get a fibred category over $\mathbf{B}$ with a cleavage. In-fact there exists a strict 2-equivalence between the 2-categories of pseudo-functors and cloven fibrations. 
\end{proposition}
\begin{proof}
	We shall not prove the proposition for the general case of a pseudo-functor but only for a strict functor since we need mostly strict functors for the work in this sequel. For this let $\Psi: \mathbf{B}^{op} \rightarrow \mathbf{Cat}$ be the strict. Now let as earlier, $I,J,K,..$ be the objects and $u,v,w,..$ be the arrows of the small category $\mathbf{B}$. Now the usual contra-variant functor diagram is given as 
	
	\begin{equation}
	\xymatrix{
		K \ar[r]^{w} \ar[dr]_{u \circ w} & I \ar[d]^{u} \\
		& J 
	} 
	\quad \\
	\xrightarrow{\Psi}
	\quad
	\xymatrix{
		\Psi(K)   & \Psi(I) \ar[l]_{\Psi(w)}  \\
		& \Psi(J) \ar[u]_{\Psi(u)} \ar[ul]^{\Psi(u \circ w)}
	}	 
	\end{equation} 
	
	Now to define a split fibration choose an $\mathbf{E}(I)$ with its second component identical (or unique up-to isomorphism) to $\Psi(I)$ and define the objects in $\mathbf{E}$ to be the pairs $(I,X_1),(I,X_2),...$ where $X_1,X_2,...$ are objects of $\Psi(I)$ and so on for all $J,K,..$. Now each functor $\Psi(u)$ assigns a well-defined unique arrow $\Psi(u)(Y_1)$ to each object $Y_1$ of $\Psi(J)$. The claim is that these arrow will act as the cartesian lifts for the fibration and therefore one can define the arrows of this category as $(u,f):(I,X_1) \rightarrow (J,Y_1)$ where $f:X_1 \rightarrow \Psi(u)(Y_1)$. To verify this claim observe Figure~\ref{fig:strictfib}.
	\begin{figure}[!h]	
		\begin{equation*}
		\xymatrix{ & (K,Z_1) \ar@{.>}[ld]_{g'} \ar@{.>}[dd]_{h} \ar[rddr]^{g} & \\  
			(I,\Psi(u \circ w)(Y_1)) \ar[rd]_{\bar{w}} \ar[rdrr]^{\overline{u \circ w}} & & \\
			& (I,\Psi(u)(Y_1)) \ar[rr]_{\bar{u}} & & (J,Y_1) \\
		}
		\end{equation*}
		\caption{Proof for the cartesian lifts in fibration formed using a strict functor $\Psi$}
		\label{fig:strictfib}
	\end{figure}
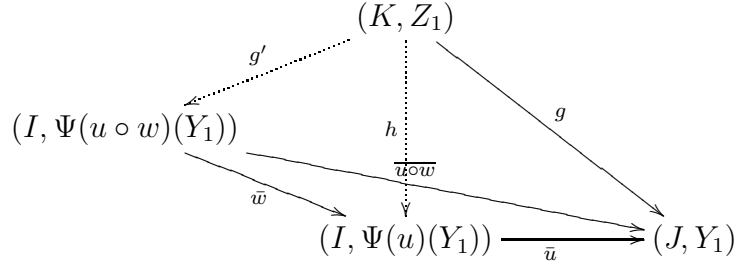

	For every arrow $g$ to $(J,Y_1)$ we have to show that there exists a unique $h$ above $w$ as per the classic cartesian lift definition as shown in the Figure~\ref{fig:strictfib2}. Now for every $g$ from some $(K,Z_1)$ to $(J,Y_1)$ there is a unique $g':Z_1 \rightarrow \Psi(u \circ w)(Y_1)$ in $\Psi(K)$ which follows from the definition of arrows in $E$ and the fact that $g$ is above $u \circ w$. since the functor $\Psi(w)$ assigns a well-defined object $\Psi(w)\Psi(u)(Y_1)$ to the object $\Psi(u)(Y_1)$ and since $\Psi$ is a true functor which implies $\Psi(w \circ u)= \Psi(w)\Psi(u)$ by its definition, therefore the arrow $\bar{w}$ is precisely same as $\Psi(w)\Psi(u)(Y_1)$. Given a unique $g'$, we have a unique $h=(w,g')$ which is the required factorization of $g$ through $\bar{u}$ proving the claim. The classic cartesian square for this case is shown below.
	\begin{figure}[!h]		
		\begin{equation*}
		\xymatrix{ & \\
			\mathbf{E} \ar[dd]_{P} \\ 
			& \\
			\mathbf{B}}
		\qquad
		\xymatrix{ (K,Z_1) \ar@{|->}[dd] \ar@{.>}[rd]_{h} \ar[rdrr]^{g} & & \\
			& (I,\Psi(u)(Y_1)) \ar@{|->}[dd] \ar[rr]_{\bar{u}} & & (J,Y_1) \ar@{|->}[dd] \\
			K \ar[rd]_{w} \ar[rdrr]^{u \circ w =Pg}  & & \\
			& I \ar[rr]_{u} & & J }
		\end{equation*}
		\caption{Classic fibration definition in the case of strict functor $\Psi$}
		\label{fig:strictfib2}
	\end{figure}
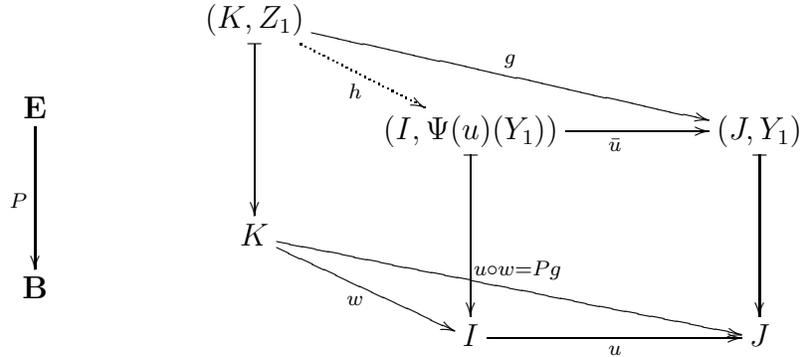
	For the reverse case we only provide a few hints and leave the relatively easy proof for the reader. For the reverse case given a split fibration the reader may observe that the cartesian lifts of a given arrow in the base for objects in the fibers sets up a pullback functor in the reverse direction. The fibers along with these pullback functors when worked out will give rise to an associated strict contravariant functor from $\mathbf{B}$ to $\mathbf{Cat}$.      
\end{proof}

For the proof in the complex case of a pseudo-functor the reader might wish to consult \cite{BJ},\cite{vistoli},\cite{barnet}.  

The definition of fibration could be interpreted intuitively in that it makes some part of the structure (that which is specified by the category $\mathbf{B}$) within $\mathbf{E}$ explicit. This will become obvious after we have explored the examples and connections with usual functor characterized by the base structured categories.

\subsection{An Example: $P: \mathbf{Vect} \rightarrow \mathbf{Fld}$}
In this section we shall understand both intuitively and rigorously that the category of all vector spaces over arbitrary fields with linear transformations is truly a fibred category with its base as the category of all fields and field homomorphisms. This example is well-known in the category theory literature. 

Since a field acts on an Abelian group; we can define a functor from $\mathbf{Fld}$ to $\mathbf{Cat}$ which maps each object $K$ to a full subcategory of $\mathbf{Ab}$ (considered as an object of $\mathbf{Cat}$) consisting of all those Abelian groups on which $K$ can act. This functor then sends field homomorphisms to functors on these subcategories. Since the action of a field on Abelian group can be defined both as a left as well as right action we can define the above functor in a contravariant fashion giving us $\mathbf{Fld}$-indexed category which is a (strict) functor $\Psi:\mathbf{Fld}^{op} \rightarrow \mathbf{Cat}$. It maps each object $K \in \mathbf{Fld}$ to the subcategory $K_{\mathbf{Ab}}$ consisting of Abelian groups on which $K$ acts.  Hence $\Psi(K)=K_{\mathbf{Ab}}$ and a morphism $u:K \rightarrow L$ is mapped to a functor $\Psi(u):L_{\mathbf{Ab}} \rightarrow K_{\mathbf{Ab}}$. 

\begin{equation*}
\xymatrix{
	u^*(A'') \ar@/_1pc/[dddd]_{u^*(g \circ f)} \ar[dd]^{u^*(f)} \ar@{.>}[rr]^{\bar{u}(A'')} &  & A'' \ar@/^1pc/[dddd]^{g \circ f} \ar[dd]_{f} \\
	& & \\
	u^*(A') \ar[dd]^{u^*(g)} \ar@{.>}[rr]^{\bar{u}(A')} &  & A' \ar[dd]_{g}  \\
	& & \\
	u^*(A) \ar@{.>}[rr]^{\bar{u}(A)} &  & A \\
	K \ar@(dl,ul)[]|{e_K} \ar[rr]_{u} &  & L \ar@(dr,ur)[]|{e_L}
}
\end{equation*} 
Since functor $\Psi$ is a strict functor it doesn't involve natural isomorphisms and is therefore split fibration example. Now Grothendieck construction $\int_{\mathbf{Fld}}\Psi$ is the total category with,
\begin{itemize}
	\item \textbf{objects} $(K,A)$ where $K \in \mathbf{Fld}$ and $A \in \Psi(K).$
	\item \textbf{morphisms} $(K,A) \rightarrow (L,A')$ are the pairs $(u,f)$ with $u:K \rightarrow L$ in $\mathbf{Fld}$ and $f:A \rightarrow u^*(A')=\Psi(u)(A')$ in $\Psi(K).$
	\item \textbf{identity} $(K,A) \rightarrow (K,A)$ is pair $(id_K,id_A)$, since $id_{\Psi(K)} = (id_K)^*$.
	\item \textbf{composition}
	$\xymatrix{
		(K,A) \ar[r]^{(u,f)} & (L,A') \ar[r]^{(v,g)} & (F,A'') 
	}$
\end{itemize}

The pairs $(K,A)$ are precisely the $K$-vector spaces with $A$, the underlying Abelian group on which the field acts as $\cdot :K \times A \rightarrow A$. The morphisms $(u,f)$ with $u$ field homomorphism and $f:A \rightarrow A'$ is an (Abelian) group homomorphism such that $f(a \cdot x)=u(a) \cdot f(x) \forall a \in K, x \in A$. Thus the fiber on $K$ is precisely the subcategory of $\mathbf{Vect}$ commonly denoted as $\mathbf{Vect}_K$ consisting of all $K$-vector spaces.   The standard diagram of fibration looks something as shown in Figure~\ref{fig:vect} where $P$ is the forgetful functor sending every vector space to its underlying field.
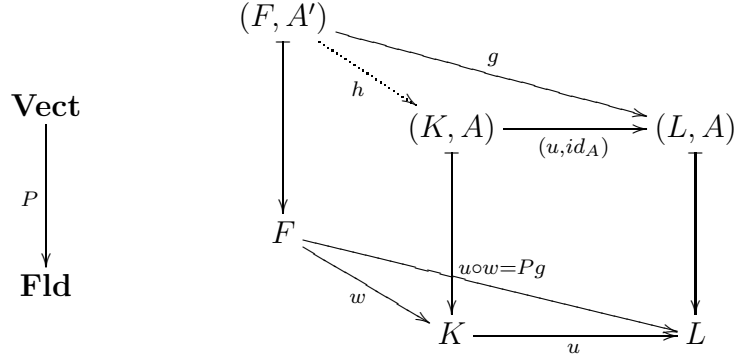
\begin{figure}
	\begin{equation*}
	\xymatrix{ & \\
		\mathbf{Vect} \ar[dd]_{P} \\ 
		& \\
		\mathbf{Fld}}
	\qquad
	\xymatrix{ (F,A') \ar@{|->}[dd] \ar@{.>}[rd]_{h} \ar[rdrr]^{g} & & \\
		& (K,A) \ar@{|->}[dd] \ar[rr]_{(u,id_A)} & & (L,A) \ar@{|->}[dd] \\
		F \ar[rd]_{w} \ar[rdrr]^{u \circ w =Pg}  & & \\
		& K \ar[rr]_{u} & & L }
	\end{equation*}
	\caption{Category of vector spaces as a fibration}
	\label{fig:vect}
\end{figure}

\subsection{Connection of $(F,\mathbf{C},\mathbf{Cat})$ to fibred category}
\label{subsec:con}
First we illustrate the precise connection between a strict functor into $\mathbf{Cat}$ to the fibred category. Consider again a general functor $F: \mathbf{C} \rightarrow \mathbf{Cat}$. As seen earlier this produces a category $(F,\mathbf{C},\mathbf{Cat})$ as shown in the Figure~\ref{fig:cat} where the image subcategory is a 2-category with categories as objects and functors as arrows.

\begin{figure}
	\begin{tikzcd}[row sep=large, column sep=30ex]
		(X,FX) \arrow[loop left]{l}{\mathrm{id}_{(X,FX)}} \arrow{dr}{(g\circ f,Fg\cdot Ff)} \arrow{r}{(f,Ff)} \arrow[swap]{d}{(h\circ g\circ f,Fh\cdot Fg\cdot Ff)} & (Y,FY) \arrow{d}{(g,Fg)} \\
		(W,FW) & (Z,FZ) \arrow{l}{(h,Fh)}
	\end{tikzcd}
	\caption{Category $(F,\mathbf{C},\mathbf{Cat})$ characterizing $F$}
	\label{fig:cat}
\end{figure}
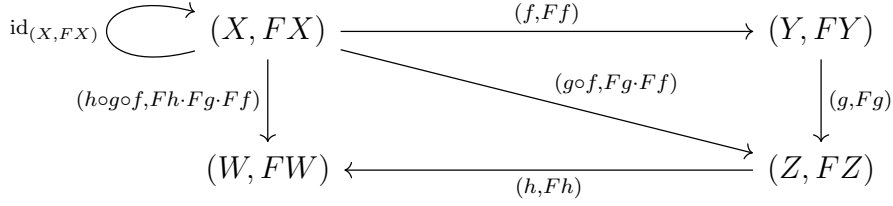

since $FX,FY,...$ are all categories and $Ff,Fg,...$ are functors it is possible to expand the second component of node and form new objects and arrows correspondingly as shown in the Figure~\ref{fig:expand}. Observe that this structurally corresponds to a fiber of a total category in Grothendieck fibration. 
\begin{figure}
	\begin{tikzcd}[row sep=large, column sep=30ex]
		(X,A1) \arrow[loop left]{l}{\mathrm{id}_{(X,A1)}} \arrow{dr}{(\mathrm{id}_X,a2\cdot a1)} \arrow{r}{(\mathrm{id}_X,a1)} & (X,A2) \arrow{d}{(\mathrm{id}_X,a2)} \\
		& (X,A3) 
	\end{tikzcd}
	\caption{Expansion of the node $(X,FX)$}
	\label{fig:expand}
\end{figure}
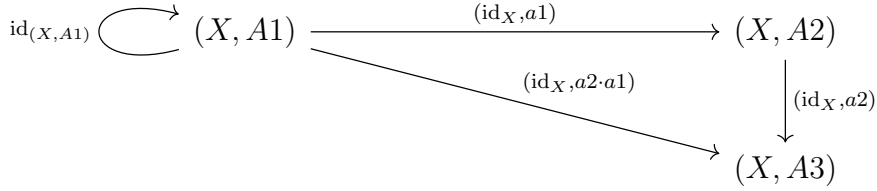

Next we repeat the above expansion procedure for the second component of every node in $(F,\mathbf{C},\mathbf{Cat})$  to get a full blown structure as shown in Figure~\ref{fig:blown} where we only two nodes corresponding to objects $X$ and $Y$ are demonstrated. Here the horizontal arrows are the restrictions of functor $(f,Ff)$ to objects.
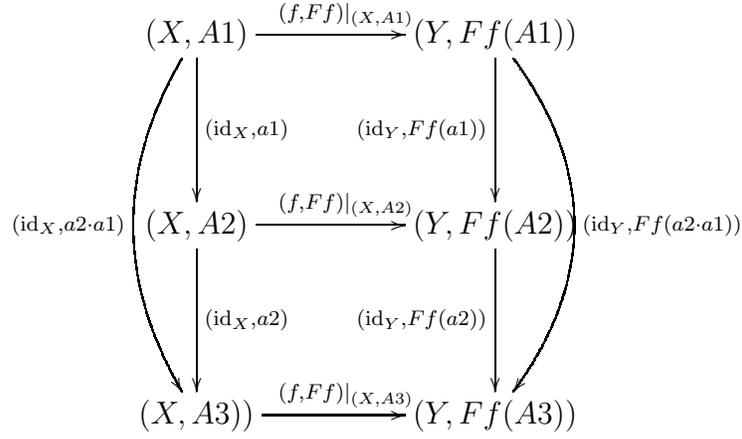
\begin{figure}
	\begin{equation*}
	\xymatrix{
		(X,A1) \ar@/_2pc/[dddd]_{(\mathrm{id}_X,a2\cdot a1)} \ar[dd]^{(\mathrm{id}_X,a1)} \ar@{>}[rr]^{(f,Ff)|_{(X,A1)}} &  & (Y,Ff(A1)) \ar@/^2.5pc/[dddd]^{(\mathrm{id}_Y,Ff(a2\cdot a1))} \ar[dd]_{(\mathrm{id}_Y,Ff(a1))} \\
		& & \\
		(X,A2) \ar[dd]^{(\mathrm{id}_X,a2)} \ar@{>}[rr]^{(f,Ff)|_{(X,A2)}} &  & (Y,Ff(A2)) \ar[dd]_{(\mathrm{id}_Y,Ff(a2))}  \\
		& & \\
		(X,A3)) \ar@{>}[rr]^{(f,Ff)|_{(X,A3)}} &  & (Y,Ff(A3))
	}
	\end{equation*}
	\caption{Schematic of the category with $(X,FX)$,$(Y,FY)$ expanded.}
	\label{fig:blown}
\end{figure}

Now using the Proposition~\ref{strict} and its proof we can convert the associated strict functor $\bar{F}: \mathbf{C}^{op} \rightarrow \mathbf{Cat}$ into an equivalent split fibration as shown in the Figure~\ref{fig:fib} where we have used Lemma~\ref{contra} and the fact $\bar{F}X = FX$. Recall that we denote opposite arrow of $f$ by $f^{\circ}$. 
\begin{figure}
	\begin{equation*}
	\xymatrix{
		(X,A1) \ar@/_2pc/[dddd]_{(\mathrm{id}_X,a2\cdot a1)} \ar[dd]^{(\mathrm{id}_X,a1)} \ar@{.>}[rr]^{\overline{f}(Y,Ff(A1))} &  & (Y,Ff(A1)) \ar@/^2.5pc/[dddd]^{(\mathrm{id}_Y,Ff(a2\cdot a1))} \ar[dd]_{(\mathrm{id}_Y,Ff(a1))} \\
		& & \\
		(X,A2) \ar[dd]^{(\mathrm{id}_X,a2)} \ar@{.>}[rr]^{\overline{f}(Y,Ff(A2))} &  & (Y,Ff(A2)) \ar[dd]_{(\mathrm{id}_Y,Ff(a1))}  \\
		& & \\
		(X,A3)) \ar@{.>}[rr]^{\overline{f}(Y,Ff(A3))} &  & (Y,Ff(A3)) \\
		X \ar@(dl,ul)[]|{\mathrm{id}_X} \ar[rr]_{f} &  & Y \ar@(dr,ur)[]|{\mathrm{id}_Y} 
	}
	\end{equation*}
	\caption{$\int_{\mathbf{C}}\bar{F}$ fibred on $\mathbf{C}$}
	\label{fig:fib}
\end{figure}
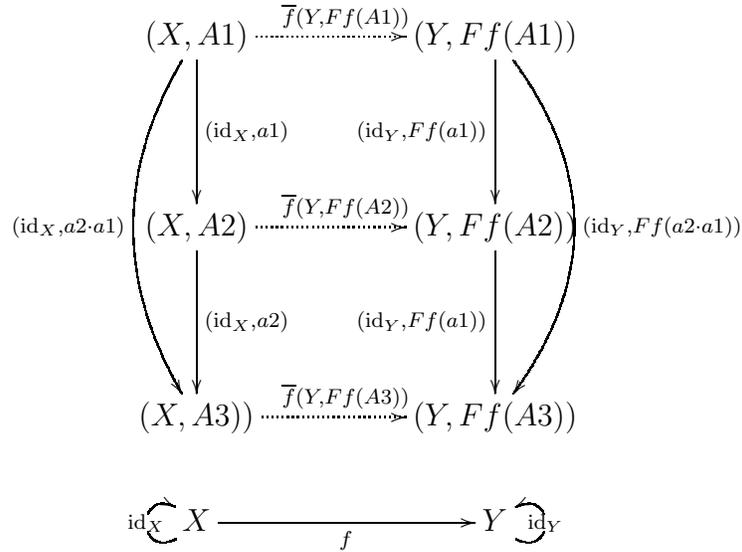

Hence we have demonstrated that since $\bar{F}$ is a strict functor,$(F,\mathbf{C},\mathbf{Cat})$ can be uniquely (up-to isomorphism) converted into a fibred category on $\mathbf{C}$ with split cleavage. The following three categories are essentially the same.

\begin{equation}
(F,\mathbf{C},\mathbf{Cat}) \cong  \int_{\mathbf{C}}\bar{F} \cong \mathcal{X} \rtimes_{F} \mathbf{C}
\end{equation}

\end{document}